\documentclass[12pt, a4paper]{article}
\usepackage{graphicx} 
\usepackage{mathrsfs}
\usepackage{amsmath}
\usepackage{amsthm}
\usepackage{appendix}
\usepackage{amssymb}
\usepackage{geometry}
\usepackage{circledsteps}
\usepackage{indentfirst}
\newtheorem{Definition}{Definition}[section]
\newtheorem{Theorem}[Definition]{Theorem}

\newtheorem{Lemma}[Definition]{Lemma}

\newtheorem*{Claim*}{Claim}
\numberwithin{equation}{section}
\allowdisplaybreaks[4]
\begin{document}

\title{$L^2$-solutions to stochastic reaction-diffusion equations with superlinear drifts  driven by space-time white noise}

\author{Shijie Shang$^{1}$, Pengyu Wang$^{1}$,
		Tusheng Zhang$^{1,2}$}
	\footnotetext[1]{\, School of Mathematics, University of Science and Technology of China, Hefei, China. Email: sjshang@ustc.edu.cn (Shijie Shang), wpy5016@mail.ustc.edu.cn (Pengyu Wang).}
	\footnotetext[2]{\, Department of Mathematics, University of Manchester, Manchester M13 9PL, United Kingdom. Email: tusheng.zhang@manchester.ac.uk }

\date{\today}

\maketitle

\begin{abstract}
\noindent
Consider the following stochastic reaction-diffusion with logarithmic superlinear coefficient $b$ driven by space-time white noise $W$,
\begin{align*}
\begin{cases}
    \partial_t u(t,x) = \frac{1}{2}\partial_{xx} u(t,x) + b(u(t,x)) + \sigma(u(t,x))W(dt,dx), \quad t >0,\ x\in [0,1].\\
    u(0,x) = u_0(x), \quad x \in [0,1],
\end{cases}
\end{align*}
where the initial value $u_0 \in L^2([0,1])$.
In this paper, we establish the existence and uniqueness of probabilistically strong solutions in $C(\mathbb{R}_{+},L^2([0,1]))$ to this equation. Our result not only resolves a recent problem posed in [Ann. Probab. 47 (2019), no. 1, 519-559], but also provides an alternative proof of the non-blowup of $L^2([0,1])$ solutions obtained in [Ann. Probab. 47 (2019), no. 1, 519–559]. Our approach crucially exploits some new Gronwall-type inequalities that we derive. Moreover, because of the nature of the nonlinearity, we are forced to work with the first order moment of the solutions, which requires precise first order moment estimates of the stochastic convolution. \\[3mm]

\end{abstract}
\section{Introduction}
Consider the following stochastic partial differential equation(SPDE):
\begin{align}\label{eq5}
\begin{cases}
    \partial_t u(t,x) = \frac{1}{2}\partial_{xx} u(t,x) + b(u(t,x)) + \sigma(u(t,x))W(dt,dx).\\
    u(t,0) = u(t,1) = 0.\\
    u(0,x) = u_0(x).
\end{cases}
\end{align}
where $t\geq0$, $ x\in [0,1]$ and $W$ is the space-time white noise on $\mathbb{R}_+ \times [0,1]$ defined on some filtrated probability space $(\Omega,\mathcal{F},\{\mathcal{F}_t\}_{t \geq 0},\mathbb{P})$ with $\mathcal{F}_t$ satisfying the usual conditions.

In this paper, assuming $u_0\in L^2([0, 1])$ we are concerned with the $L^2$-well-posedness of the SPDE (\ref{eq5}) with superlinear coefficient $b$ which are generally not dissipative and have growth like $|z|log|z|$ at infinity.

It is well known that equation \eqref{eq5} admits a unique global solution when the coefficients $b$ and $\sigma$ meet the usual Lipschitz condition. There are many papers in the literature that discuss stochastic partial differential equations with locally Lipschitz coefficients that have polynomial growth and satisfy certain monotonicity (or dissipative) conditions, see e.g. \cite{MR1961346,MR1207304,MR3410409}. However, to be able to use certain localization procedure, one has to consider the solution in a $L^{\infty}$-space with an initial value that is continuous and bounded.


If the drift $b(z)$ in \eqref{eq5} grows  faster than $|z|\log|z|$ as $z \to \infty$, then a blow-up can occur. In particular, if $b(z) = z^{1+\epsilon}$, and $\sigma = 0$ (the deterministic setting),
then equation (1.1) has blow-up solutions, but nontrivial global stationary solutions may also exist. There are many articles on this topic, see e.g. \cite{MR214914,MR1056055,MR1115140}. In \cite{MR2516340}, Bonder and Groisman studied the stochastic setting and proved that for \eqref{eq5} if $u_0 \geq 0$ is continuous, $b\geq 0$ is convex, and $\sigma$ is constant, then the Osgood condition $\int^{\infty}\frac{1}{b(z)}dz < \infty$ results in a blow-up in finite time with probability 1.
In particular, if $b(z)$ grows faster than $C|z|(\log|z|)^{1 + \epsilon}$ for some $\epsilon > 0$ as $|z| \to \infty$,
then the introduction of any amount of additive space–time white noise to equation (1.1) results in a blow-up.
So, the critical case is that the drift function $b$ grows as fast as $|u|\log|u|$ when $|u| \to \infty$. A typical example of the critical case is $b(u) = u\log|u|$.
Several papers also investigated the connections between the Osgood condition and the finite time blow-up,
see e.g. \cite{MR4177371, MR4797376, MR4904500, MR4607649, MR4486234, MR4491446,salins2025global}. In particular, in \cite{MR4177371}, the authors proved that if $b$ is nonnegative, increasing and locally Lipschitz, $\sigma$ is bounded and locally Lipschitz, the initial value $u_0$ is nonnegative and continuous, then the solutions to equation \eqref{eq5} blow up in finite time with positive probability implies that $b$ satisfies the Osgood condition. However, in all these papers, the initial value $u_0$ is assumed to be continuous or bounded, in contrast to the case $u_0 \in L^2([0,1])$ that we study in this paper.

In addition, logarithmic nonlinearity $b(z) = z\log|z|$ also appears widely in other fields. It has been introduced in the study of nonlinear wave mechanics and a relativistic field in physics,
and represents an important type of nonlinearity with a distinctive physical background.
The logarithmic wave mechanics and the logarithmic Schr\"{o}dinger equations have been studied by many authors. We refer the reader to \cite{MR426670,MR3813596,rosen1969dilatation} and the references therein for more details. The logarithmic parabolic equations have also been widely studied, see \cite{MR2425752,MR3745304,MR3263450,MR3327559} and references therein. In particular, in \cite{MR2425752} (see also \cite{MR2264255} for related works), the logarithmic parabolic (elliptic) equations were studied to understand the Ricci flow.

For the well-posedness of the stochastic reaction-diffusion equation \eqref{eq5} with logarithmic superlinear drift, in \cite{MR3909975} the authors obtained two main results. One is that when $|b(z)| {=} O(|z|\log|z|)$, $\sigma$ is bounded and $u_0 \in L^2([0,1])$, if $L^2$-valued solutions exist, then the solutions will not blow up, but the well-posedness of the solutions in $L^2([0,1])$ is left open. The other one is that if $b$ and $\sigma$ are locally Lipschitz, $|b(z)| {=} O(|z|\log|z|)$ and $|\sigma(z)| {=} o(|z|(\log|z|)^{\frac{1}{4}})$, then the solution is globally well-posed in $C(\mathbb{R}_+{\times}[0,1],\mathbb{R})$.  In \cite{MR4362370}, the authors considered the equation \eqref{eq5} driven by a Brownian motion on a bounded domain of $\mathbb{R}^d$, they established the global well-posedness of $L^2$-valued solutions when $b(z){=}z\log|z|$, $\sigma$ is locally-log-Lipschitz and $|\sigma(z)| {=} O(|z|(\log|z|)^{\frac{1}{2}})$. In \cite{MR4674636}, the authors considered equation \eqref{eq5} on the whole line, they established the global well-posedness of the solutions in $C(\mathbb{R}_{+},C_{tem})$ with initial data in $C_{tem}$ (tempered continuous functions) when $b$ is locally-log-Lipshitz and $|b(z)| {=} O(|z|\log|z|)$, $\sigma$ is Lipshitz and bounded. In \cite{MR4491446}, the authors studied the equation \eqref{eq5} on a bounded domain with bounded initial values. 
In \cite{MR4904500}, the authors also considered equation \eqref{eq5} on the whole line, they proved that if $u_0 \in L^{\infty}(\mathbb{R})\cap L^{p}(\mathbb{R})$ for $p \geq 6$, $b$ and $\sigma$ are locally Lipschitz with $b(0){=}\sigma(0){=}0$, and $|b(z)| \leq |z|\log|z|$, $|\sigma(z)| \leq |z|(\log|z|)^{\frac{1}{4}}(\log\log|z|)^{-\frac{1}{2}}$ for big $|z|$, then there exists a unique global solution. In \cite{salins2025global}, the author considered equation \eqref{eq5} on a bounded domain, they established the existence and uniqueness of global solutions under two sets of assumptions on the coefficients $b$ and $\sigma$ and assume that the initial value $u_0$ is continuous and periodic. In particular, they allowed $\sigma$ to be polynomially growing up to $|z|^{\frac{3}{2}}$. In \cite{chen2025stochasticheatequationnonlocally}, the authors considered equation \eqref{eq5} on the torus, they established  the well-posedness of positive solutions to \eqref{eq5} when the the initial value $ u_0$ is H\"older continuous and uniformly bounded away from below and above, $|b(z)| \leq |z|(\log(1/z))^{A_1}$ and $|\sigma(z)| \leq |z|(\log(1/z))^{A_2}$ with $A_1 \in (0,1)$ and $A_2 \in (0,\frac{1}{4})$ as $z$ approaches $0_{+}$, $b$ and $\sigma$ are Lipschitz continuous away from $0$. For other works related to \eqref{eq5} with logarithmic nonlinearity, the reader may refer to \cite{MR4607649,han2024supportsolutionsnonlinearstochastic,MR4881026,MR4866619,MR4805051} and references therein.

In this article, we consider the SPDE \eqref{eq5} in the $L^2$ setting with the initial value $u_0 \in L^2([0,1])$ and a driving space-time white noise. We assume that the coefficient $b$ is locally log-Lipschitz (see \eqref{wang4}) and $|b(z)| = O(|z|\log|z|)$ at infinity, the coefficient $\sigma$ is Lipschitz and of sublinear growth. Under these conditions, we obtain the global well-posedness of solutions to equation \eqref{eq5} in $L^2([0,1])$, thereby resolve the problem left open in \cite{MR3909975}. Moreover, our proof provides an alternative approach to demonstrate the non-blowup of solutions to \eqref{eq5} obtained in \cite{MR3909975}. Since equation \eqref{eq5} is driven by multiplicative space-time white noise, the It\^o's formula is not valid and therefore we can not apply the techniques used in \cite{MR4362370}. Also in the $L^2$-setting, the usual localization techniques of truncating the coefficients are not applicable. Furthermore, due to the nature of logarithmic nonlinearity, to establish uniqueness, we are forced to work with the first order moment estimates for the solutions, which is harder to handle. To address the aftermetioned problems, we derive some new types of Gronwall's inequalities and carry out delicate analysis for the logarithmic nonlinear terms, thereby successfully establishing a priori estimates for the solutions.  Meanwhile, to prove the pathwise uniquenss, we establish precise lower order moment estimates for the stochastic convolution with respect to space-time white noise.

The rest of the paper is organized as follows. In Section 2, we present the framework for equation \eqref{eq5}, give the hypothesis and state the main results. In Section 3, we will give several Gronwall-type inequalities and prove some estimates associated with the heat kernel of the Laplacian operator. In Section 4, we establish lower order moment estimates of the stochastic convolution with respect to space-time white noise and establish a priori estimates of the solutions to equation \eqref{eq5}. In Section 5, we prove the existence of probabilistically weak solutions to equation \eqref{eq5}. Section 6 is devoted to the proof of the pathwise uniqueness of solutions to \eqref{eq5}. In Appendix, we prove a convergence theorem for stochastic integrals with respect to space-time white noise.

Convention on constants. Throughout the paper, $C$ denotes a generic positive constant whose value may change from line to line. Other constants will be denoted by $C_1$, $C_2$, etc. They are all positive and their precise values are not important. The dependence of constants on parameters if needed will be indicated, e.g. $C_T$, $C_p$.
\newpage
\section{Statement of main results}
Recall the following definition.

\begin{Definition}
    An $L^2([0,1])$ valued random field solution to equation \eqref{eq5} is a jointly measurable and adapted space-time process $u := $\{$u(t,x) : (t,x) \in \mathbb{R}_{+} \times [0,1]$\} such that $u(t) \in L^2([0,1])$ $\mathbb{P}$-a.s. $\forall\ t \geq 0$ and
\end{Definition}
\begin{align}\label{eq112}
\begin{aligned}
    u(t,x) &= P_{t}u_{0}(x) + \int_{0}^{t}\int_{0}^{1}p_{t-s}(x,y)b(u(s,y))dyds \\&+ \int_{0}^{t}\int_{0}^{1}p_{t-s}
(x,y)\sigma(u(s,y))W(ds,dy),  \quad \text{a.s.} \ x\in [0,1],
\end{aligned}
\end{align}
where
\begin{align}\label{eq102}
p_{t}(x,y) := \sum_{n = 1}^{\infty}e^{-\frac{1}{2}n^2\pi^{2}t}e_{n}(x)e_{n}(y),
\end{align}
with $e_n(x) = \sqrt{2}\sin(n\pi x)$, which constitutes a set of orthonormal basis in $L^2([0,1])$, and \{$P_{t}$\}$_{t\geq0}$ is the corresponding heat semigroup in $[0,1]$ with the Dirichlet boundary condition defined by $P_tf(x) = \int_{0}^1p_{t}(x,y)f(y)dy$.
\vskip 0.3cm

\noindent\textbf{Remark 2.2.} The above mild formulation is equivalent to the weak (in the sense of partial differential equations) formulation of the stochastic reaction-diffusion equations.
\vskip 0.3cm
Next, we introduce the following conditions for the nonlinear term $b$ and $\sigma$. Set $\log_{+}(u) := \log(1\vee u)$ for any $u\geq 0.$
\\(H1). $b$ is continuous, and there exist two nonnegative constants $c_{1}$ and $c_{2}$ such that for any $u \in \mathbb{R},$
\begin{align}
    |b(u)|\leq c_{1}|u|\log_{+}|u| + c_{2}.
\end{align}
(H2). There exist nonnegative constants $c_{3}$, $c_{4}$, $c_{5}$ such that for any $u,v \in \mathbb{R},$
\begin{align}\label{wang4}
    |b(u) - b(v)| \leq c_{3}|u-v|\log_{+}\left(\frac{1}{|u-v|}\right) + c_{4}\log_{+}(|u|\vee |v|)|u-v| + c_{5}|u-v|.
\end{align}
(H3). $\sigma$ is continuous, and there exist two nonnegative constants $d_{1}$, $d_{2}$  and $\theta \in [0,1)$ such that for any $u \in \mathbb{R},$
\begin{align}
    |\sigma(u)| \leq d_{1}|u|^{\theta} + d_{2}.
\end{align}
(H4). There exists a nonnegative constant $d_{3}$ such that for any $u,v \in \mathbb{R},$
\begin{align}
    |\sigma(u) - \sigma(v)| \leq d_{3}|u-v|.
\end{align}
\textbf{Remark 2.3.} Note that condition (H2) implies condition (H1). A typical example that satisfies (H2) is $b(z) = z\log|z|$. For the global existence of the solution, only the condition (H1) is needed. The condition (H2) is used for the uniqueness of the solution.
\vskip 0.3cm
Now we can state the main results of this paper.
\setcounter{Definition}{3}
\setcounter{section}{2}
\begin{Theorem}
Assume that $u_0 \in L^2([0,1])$ and (H1) and (H3) are satisfied. Then for any $T  >0$ there exists a probabilistically weak solution to \eqref{eq5} with sample paths a.s. in $C([0,T],L^2([0,1]))$.
\end{Theorem}
\begin{Theorem}
Assume that $u_0 \in L^2([0,1])$ and (H2) and (H4) are satisfied. Then for any $T > 0$, the pathwise uniqueness holds for solutions of \eqref{eq5} in $C([0,T],L^2([0,1]))$.
\end{Theorem}
\section{Preliminaries}
\subsection{Gronwall-type inequalities}
In this subsection, we will provide three Gronwall-type inequalities, which play an important role in this paper.

The next result can be found in \cite{MR4674636}.
\begin{Lemma}
Let $X$, $a$, $c_{1}$, $c_{2}$ be nonnegative functions on $\mathbb{R}_{+}$, and let $M$ be an increasing function with $M(0) \geq 1$. Moreover, suppose that $c_{1}$, $c_{2}$ are integrable in finite time intervals. Assume that $\forall\ t\geq0$,
\begin{align*}
X(t) + a(t) \leq M(t) + \int_{0}^{t}c_{1}(s)X(s)ds + \int_{0}^{t}c_{2}(s)X(s)\log_{+}X(s)ds,
\end{align*}
and the above integral is finite. Then for any $t\geq0$,
\begin{align*}
    X(t) + a(t) \leq M(t)^{\exp(C_{2}(t))}\exp(\exp(C_{2}(t))\int_{0}^{t}c_{1}(s)\exp(-C_{2}(s))ds),
\end{align*}
where $C_{2}(t) := \int_{0}^{t}c_{2}(s)ds$.
\end{Lemma}
\begin{Lemma} Assume that $f \geq 0$ is an integrable function, and $M$ is a nonnegative increasing function, $\alpha \in [0,\frac{1}{2}]$, $c_{1}$, $c_{2}$, $c_{3}$ are nonnegative constant numbers. And $f$ satisfies the following inequality:
\begin{align*}
    f(t) \leq M(t) +  c_{1}\int_{0}^{t}f(s)ds + c_{2}\int_{0}^{t}(t-s)^{-\alpha}f(s)ds + c_{3}\int_{0}^{t}f(s)\log_{+}\frac{1}{f(s)}ds, \quad \forall\ t \geq 0.
\end{align*}
Then there exists an increasing function $C(t)$ such that for any $t \geq 0$,
\begin{align}\label{eq100}
    f(t) \leq C(t)M(t) + C(t)\int_{0}^tf(s)\log_{+}\frac{1}{f(s)}ds.
\end{align}
Furthermore, if $M \equiv 0$, then $f \equiv 0$.
\end{Lemma}
\begin{proof}
 By carrying out one iteration on the aforementioned inequality, we get
 \begin{align*}
     f(t) &\leq M(t) + c_1\int_{0}^tf(s)ds + c_2\int_{0}^t(t-s)^{-\alpha}M(s)ds \\&+ c_1c_2\int_{0}^t\int_{0}^s(t-s)^{-\alpha}f(u)duds + c_2^2\int_{0}^t\int_{0}^s(t-s)^{-\alpha}(s-u)^{-\alpha}f(u)duds \\&+ c_2c_3\int_{0}^t\int_{0}^s(t-s)^{-\alpha}f(u)\log_{+}\frac{1}{f(u)}duds + c_3\int_{0}^tf(s)\log_{+}\frac{1}{f(s)}ds.
 \end{align*}
 Using the monotonicity of $M$ and changing the order of integration gives
 \begin{align*}
     f(t) \leq C_1(t)M(t) + C_2(t)\int_{0}^tf(s)ds + C_3(t)\int_{0}^tf(s)\log_{+}\frac{1}{f(s)}ds,
 \end{align*}
 where $C_1(t)$, $C_2(t)$ and $C_3(t)$ are all positive increasing functions. By the usual Gronwall inequality,
 \begin{align*}
     f(t) \leq C_1(t)e^{C_2(t)t}M(t) + C_3(t)e^{C_2(t)t}\int_{0}^tf(s)\log_{+}\frac{1}{f(s)}ds.
 \end{align*}
 Let $C(t) = \max\left\{C_1(t)e^{C_2(t)t},C_3(t)e^{C_2(t)t}\right\}$, this proves \eqref{eq100}. In particular, when $M \equiv0$, then we have
 \begin{align*}
     f(t) \leq C(t)\int_{0}^tf(s)\log_{+}\frac{1}{f(s)}ds.
 \end{align*}
By Lemma 3.2 in \cite{MR4674636}, we get $f \equiv 0$.
 \end{proof}
The following lemma can be proven by the same method as in the proof of Lemma 3.2.
\begin{Lemma}
Let $X$ be a nonnegative function on $\mathbb{R}_{+}$, $c_1$, $c_2$, $c_3$ be positive numbers, $M$ be a nonnegative increasing function and $\alpha\in [0,\frac{1}{2}]$. Assume that for any $t \geq 0$,
\begin{align*}
X(t) \leq M(t) + c_1\int_{0}^tX(s)ds + c_2\int_{0}^tX(s)\log_{+}X(s)ds + c_3\int_{0}^t(t-s)^{-\alpha}X(s)ds,
\end{align*}
and the above integral is finite. Then for any given $T > 0$, and any $0 \leq t \leq T$,
\begin{align}\label{eq101}
    X(t) \leq \left(C_{\alpha,T}M(t) + 1\right)^{\exp(C_{\alpha,T}t)},
\end{align}
where $C_{\alpha,T}$ is a positive number depending on $\alpha$ and $T$.
\end{Lemma}

\subsection{Estimates associated with heat kernels}

In this subsection, we present some estimates associated with the heat kernel of the Laplacian operator, which will be used in our analysis later.
\begin{Lemma}
Let $p_{t}(x,y)$ be the heat kernel of Laplacian operator defined by \eqref{eq102}, then there exists a constant $C > 0$ such that for any $h > 0$,
\begin{align}
\int_{0}^{\infty}\int_{0}^{1}(p_{r+h}(x,z) - p_{r}(x,z))^{2}dzdr \leq Ch^{\frac{1}{2}}, \quad \forall\ x \in [0,1].
\end{align}
\end{Lemma}
The above lemma can be found in the Appendix of  \cite{dalang2025stochasticpartialdifferentialequations} (Lemma B.2.1). In fact, we can prove the following stronger estimate.
\begin{Lemma}
Let $p_t(x,y)$ be defined by \eqref{eq102}, then there exists a constant $C > 0$ such that for any $h > 0$,
\begin{align}
    \int_{0}^{\infty}\sup_{x\in[0,1]}\int_{0}^{1}(p_{r+h}(x,z) - p_{r}(x,z))^{2}dzdr \leq Ch^{\frac{1}{2}}.
\end{align}
\end{Lemma}
\begin{proof}
By the expression of the heat kernel and the orthonormal property of $e_n$ in $L^2([0,1])$, we have
\begin{align*}
    &\int_{0}^{\infty}\sup_{x \in [0,1]}\int_{0}^{1}(p_{r+h}(x,z) - p_{r}(x,z))^{2}dzdr \\ =& \int_{0}^{\infty}\sup_{x \in [0,1]}\int_{0}^{1}\left[\sum_{n=1}^{\infty}(e^{-\frac{1}{2}\pi^{2}n^2(r+h)}-e^{-\frac{1}{2}\pi^2n^2r})e_n(x)e_n(z)\right]^{2}dzdr\\
    =& \int_{0}^{\infty}\sup_{x \in [0,1]}\left[\sum_{n=1}^{\infty}(e_n(x))^2(e^{-\frac{1}{2}\pi^{2}n^2(r+h)} - e^{-\frac{1}{2}\pi^{2}n^2r})^{2}\right]dr \\\leq& 2\int_{0}^{\infty}\left(\sum_{n=1}^{\infty}(e^{-\frac{1}{2}\pi^{2}n^2(r+h)} - e^{-\frac{1}{2}\pi^{2}n^{2}r})^2\right)dr.
\end{align*}
Note that
\begin{align*}
    \sum_{n=1}^{\infty}(e^{-\frac{1}{2}\pi^2n^2(r+h)} - e^{-\frac{1}{2}\pi^2n^2r})^2 = \sum_{n=1}^{\infty}e^{-\pi^2n^2r}(1-e^{-\frac{1}{2}\pi^2n^2h})^2.
\end{align*}
By virtue of the following inequality:
\begin{align*}
    1 - e^{-x} \leq \min\{1, x\}, \quad \forall\ x \geq 0,
\end{align*}
we get
\begin{align*}
    2\int_{0}^{\infty}\sum_{n=1}^{\infty}e^{-\pi^2n^2r}(1 - e^{-\frac{1}{2}\pi^2n^2h})^2 dr = \frac{2}{\pi^2}\sum_{n=1}^{\infty}\frac{(1-e^{-\frac{1}{2}\pi^2n^2h})^2}{n^2} \leq \frac{2}{\pi^2}\sum_{n=1}^{\infty}\min\left\{n^{-2},\frac{1}{4}\pi^4n^2h^2\right\}.
\end{align*}
Next, we will show that there exists a constant $C > 0$, such that for any $h > 0$ the right-hand side of the above inequality can be bounded by $Ch^{\frac{1}{2}}$. The proof will be divided into two cases according to the value of $h$.\\
\textbf{Case 1}: $h \geq\frac{2}{\pi^2}$. In this case, for any $ n \geq1$, $\min\{n^{-2},\frac{1}{4}\pi^4n^2h^2\} = n^{-2}$. So
\begin{align*}
    \frac{2}{\pi^{2}}\sum_{n=1}^{\infty}\frac{1}{n^2} \leq \frac{1}{3}\leq \frac{\pi}{3\sqrt{2}}h^{\frac{1}{2}}.
\end{align*}
\textbf{Case 2}: $ 0< h< \frac{2}{\pi^{2}}$. In this case, the sum can be further divided into two parts. We have
\begin{align*}
    \frac{1}{\pi^2}\min\{n^{-2},\frac{1}{4}\pi^4n^2h^2\} &= \frac{1}{\pi^2}\sum_{n=1}^{[\frac{1}{\pi}h^{-\frac{1}{2}}]}\frac{1}{4}\pi^4n^2h^2 + \frac{1}{\pi^2}\sum_{n=[\frac{1}{\pi}h^{-\frac{1}{2}}]+1}^{\infty}n^{-2} \\&\leq \frac{1}{\pi^2}\left[\sqrt{2}\pi h^{\frac{1}{2}} + 2\int_{\frac{1}{\pi}h^{-\frac{1}{2}}}^{\infty}z^{-2}dz\right]\\
    &\leq \frac{1}{\pi^2}[\sqrt{2}\pi h^{\frac{1}{2}} + 2\pi h^{\frac{1}{2}}] \leq \frac{4}{\pi}h^{\frac{1}{2}}.
\end{align*}
So we complete the proof.
\end{proof}
\begin{Lemma}
Let $p_t(x,y)$ be the same as before. Then there exists a constant $C > 0$ such that for any $x$, $y \in [0,1]$,
\begin{align}\label{eq106}
\int_{0}^{\infty}\left(\sup_{z\in[0,1]}|p_{t}(x,z)-p_{t}(y,z)|\right)dt \leq C|x-y| + C|x-y|\log\left(\frac{1}{|x-y|}\right).
\end{align}
\end{Lemma}
\begin{proof}
By \eqref{eq102}, we have
\begin{align*}
    \sup_{z\in[0,1]}|p_{t}(x,z)-p_{t}(y,z)| &\leq C\sum_{n=1}^{\infty}|\sin(n\pi x) - \sin(n \pi y)|e^{-\frac{1}{2}n^2\pi^2 t} \\&\leq C\sum_{n=1}^{\infty}\min\{1,n|x-y|\}e^{-\frac{1}{2}n^2\pi ^2 t}.
\end{align*}
So
\begin{align}\label{eq103}
\begin{aligned}[b]
    \int_{0}^{\infty}\left(\sup_{z\in[0,1]}|p_{t}(x,z) - p_{t}(y,z)|\right)dt &\leq C\sum_{n=1}^{\infty}\frac{\min\{1,n|x-y|\}}{n^2}
    \\&\leq C|x-y|\sum_{n=1}^{[\frac{1}{|x-y|}]}\frac{1}{n} + C\sum_{n = 1 + [\frac{1}{|x-y|}]}^{\infty}\frac{1}{n^2}.
\end{aligned}
\end{align}
Note that
\begin{align*}
    \sum_{n = 1}^{[\frac{1}{|x-y|}]}\frac{1}{n} = 1 + \sum_{n = 1}^{[\frac{1}{|x-y|}]-1}\frac{1}{n+1} \leq 1 + \int_{1}^{[\frac{1}{|x-y|}]}\frac{1}{z}dz \leq 1 + \log\left(\frac{1}{|x-y|}\right).
\end{align*}
Thus, the first term of \eqref{eq103} can be bounded by
\begin{align}\label{eq104}
    C|x-y|\sum_{n=1}^{[\frac{1}{|x-y|}]}\frac{1}{n} \leq C|x-y| + C|x-y|\log\left(\frac{1}{|x-y|}\right) .
\end{align}
For the second term of \eqref{eq103}, we have
\begin{align}\label{eq105}
    C\sum_{n = 1 + [\frac{1}{|x-y|}]}^{\infty}\frac{1}{n^2} \leq C\int_{[\frac{1}{|x-y|}]}^{\infty}\frac{1}{z^2}dz  \leq C|x-y| .
\end{align}
Combining \eqref{eq103}-\eqref{eq105} together yields \eqref{eq106}.
\end{proof}
\section{Moment estimates}
In this section, we establish a priori estimates for solutions to equation \eqref{eq5}. To this end, we first establish the $L^2([0,1])$-norm moment estimate of stochastic convolution with respect to space-time white noise. In the sequel, for simplicity, we use the symbol $L^2$ to represent $L^2([0,1])$, and $\|\cdot\|_{L^2([0,1])}$ is denoted by $\|\cdot\|_{L^2}$. For clarity, sometimes we will use $\|\cdot\|_{L^2_x}$ to emphasize that the integral is with respect to the variable $x$.
\begin{Lemma}
    Let $\{\sigma(s,y):(s,y) \in \mathbb{R_+} \times [0,1]\}$ be a random
    field such that the following stochastic convolution with respect to space-time white noise is well defined. Let $\tau$ be a stopping time. Then for $ p > 8$ and $T > 0$, there exits a constant $C_{p,T} > 0$ such that
\begin{align}
    \mathbb{E}\left[\sup_{t\leq T \land \tau}\left\|\int_{0}^{t}\int_{0}^{1}p_{t-s}(x,y)\sigma(s,y)W(ds,dy)\right\|_{L^{2}_x}^p \right]\leq C_{p,T}\mathbb{E}\int_{0}^{T  \land  \tau}\|\sigma(s)\|_{L^2}^{p}ds.
\end{align}
\end{Lemma}
\begin{proof}
Let $\{e_j\}_{j=1}^{\infty}$ be an orthonormal basis of $L^2([0,1])$. Set
\begin{align*}
    W_t^j :=\int_{0}^t\int_{0}^1e_j(y)W(ds,dy).
\end{align*}
Then $\{W^j\}_{j\geq1}$ is a family of independent standard Brownian motions. Moreover, the following equality holds:
\begin{align}\label{eq107}
    \int_{0}^t\int_{0}^1p_{t-s}(x,y)\sigma(s,y)W(ds,dy) = \sum_{j = 1}^{\infty}\int_{0}^t\left(\int_{0}^1p_{t-s}(x,y)\sigma(s,y)e_j(y)dy\right)dW_s^j.
\end{align}
Set $W_s := \sum_{j = 1}^{\infty}W_s^je_j$. Then $W$ is the corresponding cylindrical Wiener process on $L^2([0,1])$. And \eqref{eq107} can be written as
\begin{align*}
    \int_{0}^t\int_{0}^1p_{t-s}(x,y)\sigma(s,y)W(ds,dy) = \int_{0}^tP_{t-s}\sigma(s)dW_s,
\end{align*}
where $P_{t-s}\sigma(s)$ is an Hilbert-Schmidt operator from $L^2([0,1])$ to $L^2([0,1])$ defined by
\begin{align*}
    \left(P_{t-s}\sigma(s)h\right)(x) := \int_{0}^1p_{t-s}(x,y)\sigma(s,y)h(y)dy, \quad \forall \ h \in L^2([0,1]),
\end{align*}
and its Hilbert-Schmidt norm can be computed as follows:
\begin{align}\label{eq108}
\begin{aligned}[b]
    \|P_{s-r}\sigma(r)\|_{HS(L^2,L^2)}^{2} &= \sum_{n=1}^{\infty}\|P_{s-r}\sigma(r)e_{n}\|_{L^2}^2 \\&= \sum_{n=1}^{\infty}\int_{0}^{1}\left(\int_{0}^{1}p_{s-r}(x,y)\sigma(r,y)e_n(y)dy\right)^2dx \\&= \int_{0}^{1}\sum_{n=1}^{\infty}\langle p_{s-r}(x,\cdot)\sigma(r,\cdot),e_n(\cdot)\rangle_{L^2}^2dx \\&=
    \int_{0}^{1}\int_{0}^{1}|p_{s-r}(x,y)\sigma(r,y)|^2dydx.
\end{aligned}
\end{align}
The proof is based on the factorization method introduced in \cite{MR1207136} and is inspired by \cite{MR4674636}. Set
\begin{align*}
    J_{\alpha}\sigma(s) = \int_{0}^s(s-r)^{-\alpha}P_{s-r}\sigma(r)dW_r,\\
    J^{\alpha-1}f(t) = \int_{0}^t(t-s)^{\alpha-1}P_{t-s}f(s)ds.
\end{align*}
Then by the stochastic Fibini theorem we obtain,
\begin{align}\label{fact2}
    \int_{0}^{t}P_{t-s}\sigma(s)dW_{s} = \frac{\sin\alpha\pi}{\pi}J^{\alpha-1}(J_{\alpha}\sigma).
\end{align}
Take $\frac{1}{p} < \alpha < \frac{1}{4} - \frac{1}{p}$, use H\"older's inequality and the BDG inequality to get
\allowdisplaybreaks
\begin{align}\label{eq1}
  \notag  &\mathbb{E}\left[\sup_{t \leq T \land \tau}\left\|\int_{0}^{t}\int_{0}^{1}p_{t-s}(x,y)\sigma(s,x)W(ds,dy)\right\|_{L^{2}_x}^p\right]
    \\=& \notag
    \mathbb{E}\left[\sup_{t \leq T \land \tau}\left\|\int_{0}^{t}(t-s)^{\alpha-1}P_{t-s}J_{\alpha}\sigma(s)ds\right\|_{L^2}^p\right] \\
    \leq&\notag
    \mathbb{E}\left[\sup_{t \leq T \land \tau}\left\{\int_{0}^{t}(t-s)^{\alpha-1}\|P_{t-s}J_{\alpha}\sigma\|_{L^2}ds \right\}^p \right] \\\leq&\notag \mathbb{E}\left[\sup_{t \leq T \land \tau}\left\{\int_{0}^{t}(t-s)^{\alpha-1}\|J_{\alpha}\sigma\|_{L^2}ds\right\}^p\right] \\
    \leq&\notag
    \mathbb{E}\left[\sup_{t \leq T \land \tau}\left\{\left[\int_{0}^{t}(t-s)^{(\alpha-1)\frac{p}{p-1}}ds\right]^{\frac{p-1}{p}} \left[\int_{0}^{t}\|J_{\alpha}\sigma\|_{L^2}^pds\right]^{\frac{1}{p}}\right\}^{p}\right] \\ \notag
    \leq& \left(\int_{0}^{T}s^{(\alpha-1)\frac{p}{p-1}}ds\right)^{p-1}\mathbb{E}\left[\int_{0}^{T}\|J_{\alpha}\sigma(s)\|_{L^2}^p1_{[0,\tau]}(s)ds\right]  \\\leq&  \notag C_{p,T}\int_{0}^{T}\mathbb{E}\left[\left\|\int_{0}^{s}(s-r)^{-\alpha}P_{s-r}\sigma(r)\mathbf{1}_{[0,\tau]}(r)dW_{r}\right\|_{L^2}^p\mathbf{1}_{[0,\tau]}(s)\right]ds \notag\\ \leq& C_{p,T}\int_{0}^{T}\mathbb{E}\left(\int_{0}^{s}(s-r)^{-2\alpha}\|P_{t-r}\sigma(r)\|_{HS(L^2,L^2)}^2\mathbf{1}_{[0,\tau]}(r)dr\right)^{\frac{p}{2}}ds.
\end{align}
By the equality \eqref{eq108} and the following inequality,
\begin{align*}
    p_{s-r}(x,y) \leq \frac{C}{\sqrt{s-r}}, \quad  \quad \forall \ x, y \in[0,1],
\end{align*}
we have
\begin{align}\label{eq109}
    \left\|P_{t-r}\sigma(r)\right\|_{HS}^2 = \int_{0}^{1}\int_{0}^{1}|p_{s-r}(x,y)\sigma(r,y)|^2dydx  \leq C(s-r)^{-\frac{1}{2}}\|\sigma(r)\|_{L^2}^2.
\end{align}
Substituting \eqref{eq109} into \eqref{eq1} gives
\begin{align*}
    &\mathbb{E}\left[\sup_{t \leq T \land \tau}\left\|\int_{0}^{t}\int_{0}^{1}p_{t-s}(x,y)\sigma(u(t,x))W(ds,dy)\right\|_{L^{2}_x}^p\right] \\ \leq&
    C_{p,T}\int_{0}^{T}\mathbb{E}\left(\int_{0}^{s}(s-r)^{-2\alpha-\frac{1}{2}}\|\sigma(r)\|_{L^2}^2\mathbf{1}_{[0,\tau]}(r)dr\right)^{\frac{p}{2}}ds \\
    \leq& C_{p,T}\int_{0}^{T}\mathbb{E}\left\{\left[\int_{0}^{s}(s-r)^{(-2\alpha-\frac{1}{2})\frac{p}{p-2}}dr\right]^{\frac{p-2}{p}}\left[\int_{0}^{s}\|\sigma(r)\|_{L^2}^{p}\mathbf{1}_{[0,\tau]}(r)dr\right]^{\frac{2}{p}}\right\}^{\frac{p}{2}}ds  \\ \leq& C_{p,T}\int_{0}^{T}\left(\int_{0}^{s}r^{-(2\alpha+\frac{1}{2})\frac{p}{p-2}}dr\right)^{\frac{p-2}{p}}\mathbb{E}\left[\int_{0}^{s}\|\sigma(r)\|_{L^2}^p\mathbf{1}_{[0,\tau]}(r)dr\right]ds \\\leq&
    C_{p,T}\mathbb{E}\int_{0}^{T \land \tau}\|\sigma(r)\|_{L^2}^pdr.
\end{align*}
\end{proof}
Having Lemma 4.1, the proof of the following result is almost identical to that of Proposition 4.2 in \cite{MR4674636}, hence we omit its proof.
\begin{Lemma}
    Let $\{\sigma(s,y): (s,y) \in \mathbb{R}_+ \times [0,1]\}$ be a random field such that the following stochastic convolution with respect to the space-time white noise is well defined. Let $\tau$ be a stopping time. Then for any $\epsilon$, $T >0$, and $0 < p \leq 8$, there exists a constant $C_{\epsilon,p,T} > 0$ such that
\begin{align*}
    &\mathbb{E}\left[\sup_{t \leq T \land \tau}\left\|\int_{0}^t\int_{0}^1p_{t-s}(x,y)\sigma(s,y)W(ds,dy)\right\|_{L^2_x}^p\right] \\\leq &\epsilon \mathbb{E}\left[\sup_{t\leq T \land \tau}\|\sigma(s)\|_{L^2}^p\right] + C_{\epsilon,p,T}\mathbb{E}\left[\int_{0}^{T \land \tau}\|\sigma(s)\|_{L^2}^pds\right].
\end{align*}
\end{Lemma}
Using the above moment estimates, we can prove the following a prior estimate of solutions.
\begin{Theorem}
Assume that (H1) and (H3) are satisfied and let
\begin{align}\label{def1}
T_0 = \min\left\{1,\frac{1}{\widetilde{C}}\log\frac{1}{\theta}\right\}
\end{align}
where $\widetilde{C}$ is a positive constant which appears in \eqref{+4}. Assume $u$ is a solution to equation \eqref{eq5}, then for any $p \geq 1$ we have,
\begin{align}
    \mathbb{E}\left[\sup_{t \leq T_0}\|u(t)\|_{L^{2}}^{p}\right]  \leq C_{p}.
\end{align}
\end{Theorem}
\begin{proof}
Without loss of generality, we assume $p > 8$. It follows from \eqref{eq112} that
\begin{align*}
    \|u(t)\|_{L^{2}} &\leq \|P_tu_0\|_{L^{2}} +  \left\|\int_{0}^{t}\int_{0}^{1}p_{t-s}(x,y)b(u(s,y))dsdy\right\|_{L^{2}_x} \\&+ \left\|\int_{0}^{t}\int_{0}^{1}p_{t-s}(x,y)\sigma(u(s,y))W(ds,dy)\right\|_{L^{2}_x}.
\end{align*}
We denote
\begin{align*}
    &I_1(t) = \left\|\int_{0}^{t}\int_{0}^{1}p_{t-s}(x,y)b(u(t,x))dsdy\right\|_{L^{2}_x}.\\
    &I_2(t) = \left\|\int_{0}^{t}\int_{0}^{1}p_{t-s}(x,y)\sigma(u(t,x))W(ds,dy)\right\|_{L^{2}_x}.
\end{align*}
For $I_1(t)$, we have
\begin{align}\label{eq113}
\begin{aligned}[b]
    I_1(t) &\leq \left\|\int_{0}^{t}\int_{0}^{1}p_{t-s}(x,y)(c_{1}|u(s,y)|\log_{+}(u(s,y)) + c_2)dsdy\right\|_{L^{2}_x} \\ &\leq c_1\left\|\int_{0}^{t}\int_{0}^{1}p_{t-s}(x,y)|u(s,y)|\log_{+}(u(s,y))dsdy\right\|_{L^{2}_x} +
    c_2t,
\end{aligned}
\end{align}
where we have used
\begin{align}\label{fact1}
    \int_{0}^1p_{t}(x,y)dy \leq  1,\quad \forall\ t > 0,\ x\in [0,1].
\end{align}
By H\"older's inequality and \eqref{fact1}, the first term on the right hand side of \eqref{eq113} can be estimated as follows,
\begin{align}\label{eq114}
\begin{aligned}[b]
    &\left\|\int_{0}^{t}\int_{0}^{1}p_{t-s}(x,y)|u(s,y)|\log_{+}(u(s,y))dsdy\right\|_{L^{2}_x} \\ \leq&
    \int_{0}^{t}\left\|\int_{0}^{1}p_{t-s}(x,y)|u(s,y)|\log_{+}(u(s,y))dy\right\|_{L^2_x}ds\\
    \leq&
    \int_{0}^{t}\left\{ \int_{0}^{1}\left(\int_{0}^{1}p_{t-s}(x,y)\log_{+}^2(|u(s,y)|dy\right)\left(\int_{0}^{1}p_{t-s}(x,y)|u(s,y)|^2dy\right)dx\right\}^{\frac{1}{2}}ds\\
    \leq&
    \int_{0}^{t}\left\{\sup_{x \in [0,1]}\int_{0}^{1}p_{t-s}(x,y)\log_{+}^2(|u(s,y)|)dy\right\}^{\frac{1}{2}}\|u(s)\|_{L^2}ds.
\end{aligned}
\end{align}
We now estimate the term
\begin{align*}
\sup_{x \in [0,1]}\int_{0}^{1}\log_{+}^2(|u(s,y)|)p_{t-s}(x,y)dy.
\end{align*}
Let $g(z) = \log_+^2z$. Then we know $g'(z) = 2\frac{1}{z}\log_+z \geq 0$. So $g(z)$ is increasing.
And $g''(z) = \frac{2}{z^2}(1- \log_+z) < 0$, provided $z > e$. So $g(z)$ is concave when $z > e$. Next, we fix any $s > 0$. Let $\Gamma := \{y\in[0,1], |u(s,y)| > e\}$ and $\mu_x(dy) := p_{t-s}(x,y)dy$. Obviously,
\begin{align}\label{eq3}
\begin{aligned}[b]
    &\sup_{x \in [0,1]}\int_{0}^{1}\log_{+}^2(|u(s,y)|)p_{t-s}(x,y)dy \\=&\sup_{x \in[0,1]}\int_{[0,1]\setminus \Gamma}g(|u(s,y)|)\mu_x(dy) + \sup_{x\in[0,1]}\int_{\Gamma}g(|u(s,y)|)\mu_x(dy).
\end{aligned}
\end{align}
The first term on the right hand side of \eqref{eq3} can be bounded as follows,
\begin{align*}
    \sup_{x \in[0,1]}\int_{[0,1]\setminus \Gamma}g(|u(s,y)|)\mu_x(dy) \leq \sup_{x \in [0,1]}\int_{0}^1p_{t-s}(x,y)dy\leq 1.
\end{align*}
Thus, by Jensen's inequality and the fact that $g$ is inceasing, we have
\allowdisplaybreaks
\begin{align}\label{eq4}
    &\sup_{x \in [0,1]}\int_{0}^{1}\log_{+}^2(|u(s,y)|)p_{t-s}(x,y)dy \notag\\ \leq& \sup_{x\in[0,1]}\left\{1 + \mu_x(\Gamma)\int_{\Gamma}g(|u(s,y)|)\frac{\mu_x(dy)}{\mu_x(\Gamma)}\right\} \notag\\\leq&
    \sup_{x\in[0,1]}\left\{1 + \mu_x(\Gamma)g\left(\frac{1}{\mu_x(\Gamma)}\int_{\Gamma}|u(s,y)|\mu_x(dy)\right)\right\} \notag\\\leq&
    \sup_{x\in[0,1]}\left\{1 + \mu_x(\Gamma)\log_{+}^2\left(\frac{1}{\mu_x(\Gamma)}\int_{0}^1|u(s,y)|p_{t-s}(x,y)dy\right)\right\}
    \notag\\ \leq&
    \sup_{x\in[0,1]}\left\{1 + \sqrt{\mu_x(\Gamma)}\log_{+}\left(\frac{1}{\mu_x(\Gamma)}\int_{0}^1|u(s,y)|p_{t-s}(x,y)dy\right)\right\}^2 \notag\\ \leq&
    \sup_{x\in[0,1]}\left\{2 + \log_{+}\left[\left(\int_{0}^1p_{t-s}(x,y)^2dy\right)^{\frac{1}{2}}\left(\int_{0}^1|u(s,y)|^2dy\right)^{\frac{1}{2}}\right]\right\}^2 \notag\\ \leq&
    \left\{2 + \log_{+}\left[(4\pi)^{-\frac{1}{4}}(t-s)^{-\frac{1}{4}}||u(s)||_{L^2}\right]\right\}^2 \notag\\ \leq&
    \left\{2 + \frac{1}{4}\log_{+}\left(\frac{1}{t-s}\right) + \log_{+}\|u(s)\|_{L^2}\right\}^2,
\end{align}
where we have used the facts that $\log_{+}(ab) \leq \log_{+}(a) + \log_{+}(b)$ for any $a$, $b \geq 0$, $\sqrt{\mu_x(\Gamma)}\log_{+}\frac{1}{\mu_x(\Gamma)} \leq 1$ and
\begin{align*}
    \int_{0}^1p_{t}(x,y)^2dy \leq (4\pi)^{-\frac{1}{2}}t^{-\frac{1}{2}},\quad  \forall\ t > 0.
\end{align*}
Combining \eqref{eq113}, \eqref{eq114} and \eqref{eq4} together gives
\begin{align*}
    I_1(t) &= \left\|\int_{0}^{t}\int_{0}^{1}p_{t-s}(x,y)b(u(t,x))dsdy\right\|_{L^{2}_x} \\&\leq
    \int_{0}^{t}c_1\left[2 + \frac{1}{4}\log_{+}\left(\frac{1}{t-s}\right) + \log_{+}\|u(s)\|_{L^2}\right]\|u(s)\|_{L^2}ds + c_2t.
\end{align*}
Hence for any $\alpha \in (0,1)$, there exists a constant $C_{c_1,\alpha}$ such that
\begin{align*}
    \|u(t)\|_{L^2} &\leq \|u_0\|_{L^2} + I_{2}(t) + c_2t + 2c_1\int_{0}^{t}\|u(s)\|_{L^2}ds \\&+ \frac{1}{4}C_{c_1,\alpha}\int_{0}^{t}(t-s)^{-\alpha}\|u(s)\|_{L^2}ds +
    c_1\int_{0}^{t}\|u(s)\|_{L^2}\log_{+}\|u(s)\|_{L^2}ds.
\end{align*}
In the sequel, we take $\alpha = \frac{1}{2}$. By Lemma 3.3 and noting that $T_0 \leq 1$, we have for any $t \leq T_0$,
\begin{align}\label{+4}
    \|u(t)\|_{L^2} \leq C(1 + c_2  + \|u_0\|_{L^2} + I_2(t))^{e^{\widetilde{C}t}},
\end{align}
Thus
\begin{align*}
    \sup_{t  \leq T_0}\|u(t)\|_{L^2} \leq C\left(1 + c_2 + \|u_0\|_{L^2} + \sup_{t\leq T_0}I_{2}(t)\right)^{e^{\widetilde{C} T_0}}.
\end{align*}
Hence we have
\begin{align}\label{+1}
    \notag&\mathbb{E}\left[\sup_{t \leq T_0}\|u(t)\|_{L^2}^p\right] \\\leq & 3^{pe^{\widetilde{C}T_0}-1}C^p\left[(1 + c_2)^{pe^{\widetilde{C}T_0}} + \|u_0\|_{L^2}^{pe^{\widetilde{C}T_0}} + \mathbb{E}\left(\sup_{t \leq T_0} I_2(t)^{pe^{\widetilde{C}T_0}}\right)\right].
\end{align}
By the Lemma 4.1, we get
\begin{align*}
    \mathbb{E}\left[\sup_{t \leq T_0}I_2(t)^{pe^{\widetilde{C}T_0}}\right] &= \mathbb{E}\left[\sup_{t \leq T_0}\left\|\int_{0}^t\int_{0}^1p_{t-s}(s,y)\sigma(u(s,y))W(ds,dy)\right\|_{L^2_x}^{pe^{\widetilde{C}T_0}}\right] \\ &\leq
    C_{T_0}\mathbb{E}\left[\int_{0}^{T_0}\|\sigma(u(s))\|_{L^2}^{pe^{\widetilde{C}T_0}}ds\right] \\ &\leq
    C_{T_0}\mathbb{E}\left[\int_{0}^{T_0}\left(\int_{0}^1c(1 + |u(s,y)|^{2\theta})dy\right)^{\frac{p}{2}e^{\widetilde{C}T_0}}ds\right] \\ &\leq
    C_{p,T_0}\left(1 + \mathbb{E}\left[\int_{0}^{T_0}\|u(s)\|_{L^2}^{p}ds\right]\right),
\end{align*}
where we have used the fact that $\theta e^{\widetilde{C}T_0} \leq 1$ and H\"older's inequality. From \eqref{+1} we deduce that for all $t \leq T_0$,
\begin{align*}
    \mathbb{E}\left(\sup_{s \leq t}\|u(s)\|_{L^2}^p\right) \leq
    C_{p,T_0} + 3^{pe^{\widetilde{C}T_0}-1}C^p\|u_0\|_{L^2}^{pe^{\widetilde{C}T_0}} + C_{p,T_0}\mathbb{E}\left[\int_{0}^{t}\|u(s)\|_{L^2}^pds\right].
\end{align*}
By the Gronwall inequality, we have
\begin{align*}
    \mathbb{E}\left(\sup_{t \leq T_0}\|u(t)\|_{L^2}^p\right) \leq C_{p,T_0}\left(1 + \|u_0\|_{L^2}^{pe^{\widetilde{C}}}\right).
\end{align*}
\end{proof}
\section{Existence of weak solutions}

In this section, we fix $T > 0$ and prove the existence of probabilistically weak solutions of equation \eqref{eq5}. We will approximate the coefficients $b$ and $\sigma$ by Lipschitz functions.

Let $\varphi$ be a nonnegative smooth function on $\mathbb{R}$ such that the support of $\varphi$ is contained in $(-1,1)$ and $\int_{\mathbb{R}}\varphi(x)dx = 1$. Let $\{\eta_n\}_{n \geq 1}$ be a sequence of symmetric smooth functions such that $0 \leq \eta_n \leq 1$, $\eta_n(x) = 1 $ if $|x| \leq n$, and $\eta_n = 0$ if $|x| \geq n +2.$
Define
\begin{align*}
    b_n(x) := n\int_{\mathbb{R}}b(y)\varphi(n(x-y))dy \times\eta_n(x).
\end{align*}
Then it is easy to check that there exist constants $L_{n}$ and $L_{b}'$ such that for any $x,y \in \mathbb{R}$, $n \in \mathbb{N}$,
\begin{align*}
    |b_n(x) - b_n(y)| \leq L_n|x-y|,\\
    |b_n(x)| \leq c_1|x|\log_{+}|x| + L_{b}'(|x|+1),
\end{align*}
where the constant $L_b'$ is independent of $n$. Moreover, if $x_n \xrightarrow{} x$ in $\mathbb{R}$, then
\begin{align}\label{wang3}
    b_n(x_n) \xrightarrow{} b(x).
\end{align}

For $n \geq 1$, consider the following stochastic partial differential equation on the interval $[0,1]$,
\begin{align*}
   u_n(t,x) &= P_{t}u_{0}(x) + \int_{0}^{t}\int_{0}^{1}p_{t-s}(x,y)b_n(u(s,y))dyds \\&+ \int_{0}^{t}\int_{0}^{1}p_{t-s}(x,y)\sigma(u_n(s,y))W(ds,dy).
\end{align*}
It is known that for each $n \geq 1$, there exists a unique solution $u_n$ to the above equation. Moreover, the sample paths of $u_n$ are
a.s. in $C(\mathbb{R}_+,L^2([0,1])).$

Following the proof of Theorem 4.3, we see that the following uniform estimate holds,
\begin{align}\label{eq6}
    \sup_{n}\mathbb{E}\left[\sup_{t\leq T_0}\|u_n(t)\|_{L^2}^p\right] \leq C_{p,T_0}
\end{align}
for any $p \geq 1$, where $T_0$ is defined by \eqref{def1}.

We now want to show that the above estimate holds for any $T > 0$.
\begin{Lemma}
    Assume that (H1) and (H3) are satisfied then for any $T > 0$ and $\forall\ p \geq 1$, the following estimate holds
\begin{align*}
    \sup_{n}\mathbb{E}\left[\sup_{t\leq T}\|u_n(t)\|_{L^2}^p\right] \leq C_{p,T}.
\end{align*}
\end{Lemma}
\begin{proof}
$\forall\ T \geq 0$, there is a nonnegative number $m$ such that $T \in [mT_0,(m+1)T_0]$, so we only need to show $\forall\ m \in \mathbb{N}$,
\begin{align*}
    \sup_{n}\mathbb{E}\left[\sup_{t\in [mT_0,(m+1)T_0]}\|u_n(t)\|_{L^2}^p\right] \leq C_{m,p}.
\end{align*}
As an example, we prove the case $m=1$, i.e.
\begin{align}\label{eq200}
    \sup_{n}\mathbb{E}\left[\sup_{t\in [T_0,2T_0]}\|u_n(t)\|_{L^2}^p\right] \leq C_{p}.
\end{align}
The same method can be applied to any $m \in \mathbb{N}$. By the uniqueness of $u_n(s,y)$, $\forall\ n \geq 0$, we have for $t \in [T_0, 2T_0]$,
\begin{align}\label{eq201}
\begin{aligned}
    u_n(t,x) &= P_{t-T_0}u_n(T_0) + \int_{T_0}^{t}\int_{0}^{1}p_{t-s}(x,y)b_n(u_n(s,y))dyds \\&+ \int_{T_0}^{t}\int_{0}^{1}p_{t-s}(x,y)\sigma(u_n(s,y))W(ds,dy).
\end{aligned}
\end{align}
Let $W^{T_0}(t,x) := W(t+T_0,x) - W(T_0,x)$, $\mathcal{F}_t^{T_0}:=\mathcal{F}_{t+T_0}$, and $u_n^{T_0}(t,x) := u_n(t+T_0,x)$, $\forall\ t \geq 0$.
Then $W^{T_0}(t,x)$ is a Brownian sheet with respect to the filtration $(\mathcal{F}^{T_0}_{t})_{t\geq0}$. Then by changing of variables, \eqref{eq201} can be rewritten as
\begin{align*}
    u_n^{T_0}(t,x) &= P_tu_n^{T_0}(0,x) + \int_{0}^{t}\int_{0}^{1}p_{t-s}(x,y)b_n(u^{T_0}_n(s,y))dsdy \\&+ \int_{0}^{t}\int_{0}^{1}p_{t-s}(x,y)\sigma(u_n^{T_0}(s,y))W^{T_0}(ds,dy),
\end{align*}
where $t \in [0, T_0]$. By the contractivity of the heat kernel and the similar argument as in the proof of Theorem 4.3, we obtain that
\begin{align*}
    \sup_{n}\mathbb{E}\left[\sup_{t\in [T_0,2T_0]}\|u_n(t)\|_{L^2}^p\right] = \sup_{n}\mathbb{E}\left[\sup_{t \leq T_0}\|u_n^{T_0}(t)\|_{L^2}^p\right] \leq C_{p}\left(1 + \sup_{n}\mathbb{E}\left[\|u_n(T_0)\|_{L^2}^{pe^{C_{c_1}}}\right]\right)
\end{align*}
for any $p \geq 1$. Since \eqref{eq6} holds for any $p \geq 1$, we see that
\begin{align*}
    \sup_{n}\mathbb{E}\left[\sup_{t\in [T_0,2T_0]}\|u_n(t)\|_{L^2}^p\right] \leq C_{p},
\end{align*}
which proves \eqref{eq200}.
\end{proof}
We let $L_b$ be a constant independent of $n$ such that
\begin{align*}
    |b_n(x)| \leq L_b(|x|\log_{+}|x| + 1).
\end{align*}
Such a constant always exists because of the assumption on the coefficient $b$.

Next, we will establish the tightness of the approximating solutions $u_n$ in $C([0,T],L^2([0,1]))$.
Set
\begin{align*}
    I_n(t,x) := \int_{0}^{t}\int_{0}^{1}p_{t-s}
(x,y)b_{n}(u_n(s,y))dyds
\end{align*}
and
\begin{align*}
    J_n(t,x) := \int_{0}^t\int_{0}^1p_{t-s}(x,y)\sigma(u_n(s,y))W(ds,dy).
\end{align*}
To prove the tightness of $u_n$ in $C([0,T],L^2([0,1]))$, it suffices to prove that both $I_n(t,x)$ and $J_n(t,x)$ are tight in $C([0,T],L^2([0,1]))$. This will be done in Lemma 5.2 and Lemma 5.3 below.
\begin{Lemma}
The sequence $I_n(t,x)$ is tight in $C([0,T],L^2([0,1]))$.
\end{Lemma}
\begin{proof}
Since $C^{\alpha}([0,1])$ is compactly embedded into $L^2([0,1])$, to establish the tightness in $C([0,T],L^2([0,1]))$, it suffices to show that:
\begin{align*}
 \lim\limits_{M \to \infty}\sup_{n}\mathbb{P}(\|I_n(t)\|_{C^{\alpha}([0,1])} \geq M) = 0,\quad \quad \forall\ t \in (0,T],
\end{align*}
where $\|\cdot\|_{C^{\alpha}([0,1])}$ is the norm of the $\alpha$-H\"older space, and
\begin{align*}
\forall\ \epsilon >0,\quad\quad \lim\limits_{h \to 0}\sup_n\mathbb{P}\left(\max_{t,s \in[0,T],|t-s| \leq h}\|I_n(t,\cdot) - I_n(s,\cdot)\|_{L^2} > \epsilon\right) = 0.
\end{align*}
Noticing that $I_n(t,0) = 0$, to verify these two conditions, it is sufficient to prove the following claims:\\
\textbf{Claim 1}: $\forall\ t \in [0,T],$ $\exists\ p,q >0,$ $C_{p,q,T} >0$, such that
\begin{align*}
\sup_n\mathbb{E}[|I_n(t,x) - I_n(t,y)|^p] \leq C_{p,q,T}|x-y|^{1+q},\quad  \forall\ x,y \in [0,1].
\end{align*}
\textbf{Claim 2}: $\exists\ p, q >0,$ and $C_{T} > 0, $ such that
\begin{align*}
\sup_n\mathbb{E}[\|I_n(t,\cdot) - I_n(s,\cdot)\|_{L^2}^p] \leq C_{T}|t-s|^{1+q}, \quad \forall\ s,t \in [0,T].
\end{align*}
We now prove \textbf{Claim 1}. Without loss of generality, we assume that
\begin{align*}
    |b_n(z)| \leq L_b(1 + |z|^2).
\end{align*}
We have
\begin{align}\label{eq202}
\begin{aligned}[b]
    |I_n(t,x) - I_n(t,y)|  &= \left|\int_{0}^t\int_{0}^1(p_{t-s}(x,z) - p_{t-s}(y,z))b_n(u_n(s,z))dsdz\right| \\ &\leq
    \int_{0}^t\int_{0}^1|p_{t-s}(x,z) - p_{t-s}(y,z)||b_n(u_n(s,z))|dsdz \\ &\leq
    L_b\left(\int_{0}^t\int_{0}^1|p_{t-s}(x,z) - p_{t-s}(y,z)||u_n(s,z)|^2dsdz\right) \\ &+
    L_b\int_{0}^{t}\int_{0}^1|p_{t-s}(x,z) - p_{t-s}(y,z)|dsdz.
\end{aligned}
\end{align}
The first term on the right hand side of above can be bounded as follows,
\begin{align*}
    &\left(\int_{0}^t\int_{0}^1|p_{t-s}(x,z) - p_{t-s}(y,z)||u_n(s,z)|^2dsdz\right) \\\leq& \left(\max _{s \leq T}\int_{0}^1|u_n(s,z)|^2dz\right)\int_{0}^t\max_{z\in[0,1]}\left|p_{t-s}(x,z) - p_{t-s}(y,z)\right|ds  \\ \leq&
    C\left(\max_{s \leq T}\int_{0}^1|u_n(s,z)|^2dz\right)|x-y|\left[\log\left(\frac{1}{|x-y|}\right) + 1\right],
\end{align*}
where we have used Lemma 3.6. Note that $\forall\ \gamma \in (0,1)$, there exist nonnegative constants $c_1$ and $c_2$ such that
\begin{align*}
    \log\left(\frac{1}{|x-y|}\right) \leq c_1\frac{1}{|x-y|^{\gamma}}  + c_2.
\end{align*}
So we have
\begin{align*}
    &\int_{0}^t\int_{0}^1|p_{t-s}(x,z) - p_{t-s}(y,z)||u_n(s,z)|^2dsdz \\\leq& C\left(\max_{s\leq T}\int_{0}^1|u_n(s,z)|^2dz\right)(c_1|x-y|^{1-\gamma} + c_2|x-y|) \\ \leq&
    C\left(\max_{s\leq T}\int_{0}^1|u_n(s,z)|^2dz + 1\right)|x-y|^{1-\gamma}.
\end{align*}
The second term on the right hand side of \eqref{eq202} can be bounded similarly, but much simpler. Hence
\begin{align*}
    |I_n(t,x) - I_n(t,y)| &\leq C\left(\max_{s \leq T}\|u_n(s)\|_{L^2}^{2} + 1\right)|x-y|^{1-\gamma} + L_b\int_{0}^{t}\int_{0}^1|p_{t-s}(x,z) - p_{t-s}(y,z)|dsdz \\
    &\leq C\left(\max_{s \leq T}\|u_n(s)\|_{L^2}^{2} + 2\right)|x-y|^{1-\gamma}.
\end{align*}
So
\begin{align*}
    \sup_n\mathbb{E}\left[|I_n(t,x) - I_n(t,y)|^p\right] \leq C\sup_n\mathbb{E}\left[\max_{s\leq T}\|u_n(s)\|_{L^2}^{2} + 2\right]^p|x-y|^{(1-\gamma)p}.
\end{align*}
By choosing $p$ such that $(1-\gamma)p >1$ and noting Lemma 5.1, \textbf{Claim 1} is proved. Next we prove \textbf{Claim 2}. Without loss of generality, we assume $s < t$. Obviously
\begin{align}\label{eq203}
\begin{aligned}[b]
    \|I_n(t) - I_n(s)\|_{L^2}^p & \leq
    c_p\left\|\int_{0}^s\int_{0}^1(p_{t-r}(x,y)-p_{s-r}(x,y))b_n(u_n(r,y))drdy\right\|_{L^2_x}^p \\&+ c_p\left\|\int_{s}^t\int_{0}^1p_{t-r}(x,y)b_n(u_n(r,y))drdy\right\|_{L^2_x}^p \\ &:=
    c_p(A_1^p + A_2^p).
\end{aligned}
\end{align}
Now we estimate the terms $A_1^p$ and $A_2^p$. For $A_1$, we have
\begin{align*}
    A_1 &\leq
    L_b\left(\int_{0}^1\int_{0}^s\|p_{t-r}(x,y) - p_{s-r}(x,y)\|_{L^2_x}|u_n(r,y)|^2drdy\right)
    \\ &+
    L_b\int_{0}^1\int_{0}^s\|p_{t-r}(x,y) - p_{s-r}(x,y)\|_{L^2_x}drdy.
\end{align*}
Now
\begin{align*}
    &\int_{0}^1\int_{0}^s\|p_{t-r}(x,y) - p_{s-r}(x,y)\|_{L^2_x}|u_n(r,y)|^2drdy \\\leq& \left(\int_{0}^s\sup_{y}\|p_{t-r}(x,y) - p_{s-r}(x,y)\|_{L^2_x}dr\right) \times \left(\max_{r \leq T}\int_{0}^1|u_n(r,y)|^2dy\right) \\
    \leq&
    \left(\int_{0}^s\sup_y\|p_{t-r}(x,y) - p_{s-r}(x,y)\|_{L^2_x}^2dr\right)^{\frac{1}{2}} \times \left(\max_{r \leq T}\int_{0}^1|u_n(r,y)|^2dy\right)
    \\ \leq&
    C(t-s)^{\frac{1}{4}}\left(\max_{r \leq T}\int_{0}^1|u_n(r,y)|^2dy\right),
\end{align*}
where Lemma 3.5 is used in the last line. Therefore,
\begin{align}\label{eq204}
    A_1^{p} \leq C_{T}\left(\max_{r \leq T}\|u_n(r)\|_{L^2}^{2p} + 1\right)(t-s)^{\frac{p}{4}}.
\end{align}
For $A_2^p,$ we have
\begin{align}\label{eq205}
\begin{aligned}[b]
    &\left\|\int_{s}^t\int_{0}^1p_{t-r}(x,y)b_n(u_n(r,y))drdy\right\|_{L^2_x}^p \\\leq& \left(\int_{s}^t\int_{0}^1\|p_{t-r}(x,y)\|_{L^2}|b_n(u_n(r,y))|drdy\right)^p \\ \leq&
    \left(L_b\int_{s}^t\int_{0}^1(t-r)^{-\frac{1}{4}}(|u_n(r,y)|^2 + 1)drdy\right)^p  \\ \leq& C_p\left(\sup_{r \leq T}\|u_n(r)\|_{L^2}^{2p} + 1\right)(t-s)^{\frac{3}{4}p}.
\end{aligned}
\end{align}
Combining \eqref{eq203}-\eqref{eq205} together and then taking expectations gives
\begin{align*}
    \sup_{n}\mathbb{E}[\|I_n(t) - I_n(s)\|_{L^2}^p] \leq C_{T}\sup_n\mathbb{E}\left[\max_{r \leq T}\|u_n(r)\|_{L^2}^{2p} + 1\right](t-s)^{\frac{p}{4}}, \quad \forall\ s,t \in [0,T].
\end{align*}
We choose $p > 4$, then the \textbf{Claim 2} is proved by Lemma 5.1.
\end{proof}
Next, we will prove the tightness of $\{J_n\}_{n \geq 1}$ with a different method.
\begin{Lemma}
The sequence $J_n(t,x)$ is tight in $C([0,T],L^2([0,1]))$.
\end{Lemma}
\begin{proof}
As in \eqref{fact2}, we use the factorization method to write
\begin{align*}
    J_n(t) = \frac{\sin\alpha\pi}{\pi}J^{\alpha-1}(J_{\alpha} \sigma(u_n))(t).
\end{align*}
According to the fact (see Proposition 8.4 in \cite{MR1207136}) that the embedding
\begin{align*}
    J^{\alpha-1}: L^p([0,T],L^2([0,1])) \to C([0,T],L^2([0,1]))
\end{align*}
is compact provided $\frac{1}{p} < \alpha$. Thus to prove the tightness of $J_n$ in $C([0,T],L^2([0,1]))$, it suffices to prove that for some $\frac{1}{p} < \alpha$, $\forall\ \epsilon > 0$, there exists a constant $M > 0$, such that
\begin{align*}
    \sup_n\mathbb{P}(\|J_{\alpha}\sigma\|_{L^p([0,T],L^2)}  > M)  < \epsilon.
\end{align*}
By the Chebyshev inequality, it is sufficient to show
\begin{align*}
    \sup_n\int_{0}^T\mathbb{E}[\|J_{\alpha}\sigma(s)\|_{L^2}^p]ds = \sup_n\int_{0}^{T}\mathbb{E}\left[\left\|\int_{0}^s(s-r)^{-\alpha}P_{s-r}\sigma(u_n(r))dW_r\right\|_{L^2}^p\right]ds < \infty.
\end{align*}
We use the BDG inequality and \eqref{eq108} to deduce
\begin{align*}
    &\sup_n\int_{0}^{T}\mathbb{E}\left[\left\|\int_{0}^s(s-r)^{-\alpha}P_{s-r}\sigma(u_n(r))dW_r\right\|_{L^2}^p\right]ds \\ \leq& \sup_n\int_{0}^{T}\mathbb{E}\left(\int_{0}^s(s-r)^{-2\alpha}\|P_{t-r}\sigma(u_n(r))\|_{HS(L^2,L^2)}^2dr\right)^{\frac{p}{2}}ds\\
    \leq& C_p\sup_n\int_{0}^T\mathbb{E}\left(\int_{0}^s(s-r)^{-2\alpha-\frac{1}{2}}\|\sigma(u_n(r))\|_{L^2}^2dr\right)^{\frac{p}{2}}ds \\=& C_p\sup_{n}\int_{0}^T\left\|\int_{0}^s(s-r)^{-2\alpha-\frac{1}{2}}\|\sigma(u_n(r))\|_{L^2}^2dr\right\|_{L^{\frac{p}{2}}(\Omega)}^{\frac{p}{2}}ds \\ \leq&C_p \sup_{n}\int_{0}^T
    \left(\int_{0}^s(s-r)^{-2\alpha - \frac{1}{2}}(\mathbb{E}\|\sigma(u_n(r))\|_{L^2}^p)^{\frac{2}{p}}dr\right)^{\frac{p}{2}}ds \\ \leq&C_p \sup_{n}\int_{0}^T\left(\sup_{r \in [0,T]}\mathbb{E}\|\sigma(u_n(r))\|_{L^2}^p\right)\left(\int_{0}^s(s-r)^{-2\alpha-\frac{1}{2}}dr\right)^{\frac{p}{2}}ds\\ \leq& C_p\sup_{n}\int_{0}^T\left(\sup_{r \in [0,T]}\mathbb{E}\|u_n(r)\|_{L^2}^p+1\right)\left(\int_{0}^s(s-r)^{-2\alpha-\frac{1}{2}}dr\right)^{\frac{p}{2}}ds.
\end{align*}
If we take $p > 4$ so that $\frac{1}{p}  < \alpha < \frac{1}{4}$, then the right-hand side of the above inequality is finite by Theorem 5.1. So we complete the proof.
\end{proof}
Now we prove the existence of probabilistically weak solutions to equation \eqref{eq5}.
\begin{proof}[Proof of Theorem 2.4]
 Let $C([0,T] {\times}[0,1],\mathbb{R})$ be the space of real-valued continuous functions on $[0,T] \times [0,1]$. Set $\Upsilon := C([0,T],L^2([0,1])) {\times} C([0,T]{\times}[0,1],\mathbb{R})$. By Lemma 5.2 and Lemma 5.3, the family of laws $\mathcal{L}(u_n,W)$ of the random vectors $(u_n,W)$ is tight in $\Upsilon$, where $W$ (with a slight abuse of notations) is the Brownian sheet corresponding to the space-time noise. By Prokhorov's theorem and  Skorokhod's representation theorem, there exist a new probability space $(\widetilde{\Omega},\widetilde{\mathcal{F}},\widetilde{\mathbb{P}})$, a sequence of of $\Upsilon$-valued random vectors $\{(\tilde{u}_n,\widetilde{W}_n)\}$ and a $\Upsilon$-valued random vector $(\tilde{u},\widetilde{W})$ such that $\mathcal{L}(\tilde{u}_n,\widetilde{W}_n) = \mathcal{L}(u_n,W)$, $\mathcal{L}(\widetilde{W}) = \mathcal{L}(W)$ and
\begin{align}\label{guo2}
\sup_{t\leq T}\|\tilde{u}_n - \tilde{u}\|_{L^2} + \sup_{t \geq 0,x\in[0,1]}|\widetilde{W}_n(t,x) - \widetilde{W}(t,x)| \xrightarrow{n \to \infty}0.
\end{align}
Let $\widetilde{\mathcal{F}}_t$ be the filtration satisfying the usual conditions and generated by
\begin{align*}
    \{\tilde{u}(s),\widetilde{W}(s,x):s \leq t,x\in \mathbb{R}\}.
\end{align*}
Then $\widetilde{W}$ is an $\{\widetilde{\mathcal{F}}_t\}$ -Brownian sheet. And let $\widetilde{\mathcal{F}}_t^n$ be the filtration satisfying the usual conditions and generated by
\begin{align*}
    \{\tilde{u}_n(s),\widetilde{W}_n(s,x):s \leq t,x\in \mathbb{R}\}.
\end{align*}
Then $\widetilde{W}_n$ is an $\{\widetilde{\mathcal{F}}_{t}^n\}$-Brownian sheet.
From the equation satisfied by the random vector $(u_n,W)$, it follows that
\begin{align}\label{guo1}
\begin{aligned}
    \tilde{u}_n(t,x) &= P_t\tilde{u}_0(x) + \int_{0}^t\int_{0}^1p_{t-s}(x,y)b_n(\tilde{u}_n(s,y))dsdy \\&+ \int_{0}^t\int_{0}^1p_{t-s}(x,y)\sigma(\tilde{u}_n(s,y))\widetilde{W}_n(ds,dy),
\end{aligned}
\end{align}
where for each $n \geq 1$, $\widetilde{W}_n(ds,dy)$ is the space-time white noise corresponding to the Brownian sheet $\widetilde{W}_n$. We will let $n$ tend to $\infty$ in \eqref{guo1} to show that $(\tilde{u},\widetilde{W})$ is a solution to equation \eqref{eq5}. First, for any $t \leq T$, we have
\begin{align*}
    &\sup_{n}\mathbb{E}\left[\int_{0}^t\int_{0}^1\|p_{t-s}(x,y)\|_{L_x^2}^{\frac{3}{2}}|b_n(\tilde{u}_n(s,y))|^{\frac{3}{2}}dyds\right]\\ \leq
    &c_1^{\frac{3}{2}}\sup_{n}\mathbb{E}\left[\int_{0}^t\int_{0}^1\|p_{t-s}(x,y)\|_{L^2_x}^{\frac{3}{2}}|\tilde{u}_n(s,y)|^{\frac{3}{2}}\log_{+}^{\frac{3}{2}}|\tilde{u}_n(s,y)|dyds\right] + C_{T,c_2}\\ \leq &
    2c_1^{\frac{3}{2}}\sup_{n}\mathbb{E}\left[\int_{0}^t\int_{0}^1\|p_{t-s}(x,y)\|_{L^2_x}^{\frac{3}{2}}|\tilde{u}_n(s,y)|^2dyds\right] + C_{T,c_2} \\ \leq &
    C_{T,c_1}\sup_{n}\mathbb{E}\left[\sup_{s \leq T}\|\tilde{u}_n(s)\|_{L^2}^2\right] + C_{T,c_2} < \infty,
\end{align*}
where we have used the inequality
\begin{align*}
\log_{+}|u| \leq 2|u|^{\frac{1}{3}}, \quad \forall\ u \in \mathbb{R}.
\end{align*}
So the sequence
$\{\|p_{t-s}(x,y)\|_{L^2_x}|b_n(\tilde{u}_n(s,y))|\}_{n=1}^{\infty}$
is uniformly integrable on $\Omega\times [0,t] \times[0,1]$. Hence by \eqref{wang3}, \eqref{guo2} and Vitali's convergence theorem, we have for any $t \leq T$,
\begin{align*}
    &\mathbb{E}\left[\left\|\int_{0}^t\int_{0}^1p_{t-s}(x,y)b_n(\tilde{u}_n(s,y))dyds - \int_{0}^t\int_{0}^1p_{t-s}(x,y)b(\tilde{u}(s,y))dyds\right\|_{L^2_x}\right] \\ \leq &
    \mathbb{E}\left[\int_{0}^t\int_{0}^1\|p_{t-s}(x,y)\|_{L^2_x}|b_n(\tilde{u}_n(s,y)) - b(\tilde{u}(s,y))|dyds\right] \xrightarrow{n \to \infty} 0.
\end{align*}
To prove the convergence of the stochastic integral term in \eqref{guo1}, we will show that $f_n(s,y):=p_{t-s}(x,y)\sigma(\tilde{u}_n(s,y))$, $f(s,y):=p_{t-s}(x,y)\sigma(\tilde{u}(s,y))$ satisfy the conditions in Theorem 7.1 in the appendix. First it is easy to see that $f_n$, $f$ satisfy \eqref{kk1} and \eqref{kk2} respectively. And for any $1 < \eta < \min\left\{\frac{1}{\theta},\frac{3}{2}\right\}$, we have
\begin{align*}
    &\sup_{n}\mathbb{E}\left[\int_{0}^t\int_{0}^1\int_{0}^1|p_{t-s}(x,y)\sigma(\tilde{u}_n(s,y))|^{2\eta}dxdyds\right] \\\leq &\sup_{n}2^{2\eta-1}d_1^{2\eta}\mathbb{E}\left[\int_{0}^t\int_{0}^1\int_{0}^1|p_{t-s}(x,y)|^{2\eta}|\tilde{u}_n(s,y)|^{2\theta\eta}dxdyds\right] \\&+2^{2\eta-1} d_2^{2\eta}\mathbb{E}\left[\int_{0}^t\int_{0}^1\int_{0}^1|p_{t-s}(x,y)|^{2\eta}dxdyds \right] \\\leq&
    \sup_{n}C_{T,d_1,\eta}\mathbb{E}\left[\sup_{s \leq T}\int_{0}^1(|\tilde{u}_n(s,y)|^2+1)dy \right] + C_{T,d_2,\eta} \\ =& \sup_{n}C_{T,d_1,\eta}\mathbb{E}\left[\sup_{s \leq T}\|\tilde{u}_n(s)\|_{L^2}^2\right] + C_{T,d_1,d_2,\eta} < \infty,
\end{align*}
where we have used (H3) and the following fact,
\begin{align*}
    \int_{0}^t\int_{0}^1p_{t-s}(x,y)^{2\eta}dxds \leq C\int_{0}^t(t-s)^{-(\eta-\frac{1}{2})}ds < \infty.
\end{align*}
So the sequence $\{|p_{t-s}(x,y)\sigma(\tilde{u}_n(s,y))|^2\}_{n=1}^{\infty}$ is uniformly integrable on $\Omega\times[0,t]\times[0,1]\times[0,1]$. Hence by the continuity of $\sigma$ and Vitali's convergence theorem, we have for any $t \leq T$,
\begin{align*}
    &\mathbb{E}\left[\int_{0}^t\int_{0}^1\|p_{t-s}(x,y)\sigma(\tilde{u}_n(s,y)) - p_{t-s}(x,y)\sigma(\tilde{u}(s,y))\|_{L^2_x}^2dyds\right] \\ =&\mathbb{E}\left[\int_{0}^t\int_{0}^1\int_{0}^1p_{t-s}(x,y)^2|\sigma(\tilde{u}_n(s,y))-\sigma(\tilde{u}(s,y))|^2dxdyds\right] \xrightarrow{n\to \infty}0.
\end{align*}
This implies that \eqref{xxx1} holds. So we can apply Theorem 7.1 to conclude that for any $ t \leq  T$,
\begin{align*}
    \int_{0}^t\int_{0}^1p_{t-s}(x,y)\sigma(\tilde{u}_n(s,y))\widetilde{W}_n(ds,dy) \xrightarrow{\mathbb{P}} \int_{0}^t\int_{0}^1p_{t-s}(x,y)\sigma(\tilde{u}(s,y))\widetilde{W}(ds,dy)
\end{align*}
in $L^2([0,1])$ as $ n\to \infty$. Taking $n \to \infty$ in \eqref{guo1}, we see that for any $t \leq T$,
\begin{align*}
    \tilde{u}(t) = P_{t}\tilde{u}_0 + \int_{0}^t\int_{0}^1p_{t-s}(\cdot,y)b(\tilde{u}(s,y))dsdy + \int_{0}^t\int_{0}^1p_{t-s}(\cdot,y)\sigma(\tilde{u}(s,y))\widetilde{W}(ds,dy).
\end{align*}
And by \eqref{guo2} we have $\tilde{u} \in C([0,T],L^2([0,1]))$. So $\tilde{u}$ is a probabilistically weak solution of \eqref{eq5}.
\end{proof}
\section{Pathwise uniqueness}
In this section, we prove the pathwise uniqueness of solutions to \eqref{eq5} and hence obtain a unique probabilistically strong solution.
\begin{proof}[Proof of Theorem 2.5]
Let $u$, $v$ be two solutions of equation \eqref{eq5}. Then we have
\begin{align*}
    u(t,x) - v(t,x) &= \int_{0}^t\int_{0}^1p_{t-s}(x,y)(b(u(s,y)) - b(v(s,y)))dyds \\&+ \int_{0}^t\int_{0}^1p_{t-s}(x,y)(\sigma(u(s,y)) - \sigma(v(s,y)))W(ds,dy).
\end{align*}
Set $z(t,x) := u(t,x) - v(t,x)$. Then
\begin{align*}
    \|z(t)\|_{L^2}  &\leq \left\|\int_{0}^t\int_{0}^1p_{t-s}(x,y)|b(u(s,y)) - b(v(s,y))|dyds\right\|_{L_x^2}
    \\ &+
    \left\|\int_{0}^t\int_{0}^1p_{t-s}(x,y)(\sigma(u(s,y) ) -\sigma(v(s,y)) )W(ds,dy) \right\|_{L^2_x} \\ &:= I(t) + J(t).
\end{align*}
By (H2),
\begin{align*}
    I(t) &\leq \left\|c_1\int_{0}^t\int_{0}^1p_{t-s}(x,y)|u(s,y) - v(s,y)|\log_+\frac{1}{|u(s,y - v(s,y))|}dyds\right\|_{L_x^2}  \\ &+
    \left\|c_2\int_{0}^t\int_{0}^1p_{t-s}(x,y)\log_+(|u(s,y)|\vee |v(s,y)|)|u(s,y) - v(s,y))|dyds\right\|_{L_x^2}  \\ &+
    \left\|c_3\int_{0}^t\int_{0}^1p_{t-s}(x,y)|u(s,y) - v(s,y )|dyds\right\|_{L_x^2}  \\& :=
    I_1(t) + I_{2}(t) + I_{3}(t).
\end{align*}
We first estimate $I_3(t)$ as follows. By the contractivity of the heat semigroup $P_t$, we have
\begin{align}\label{estimate1}
    I_3(t) &\leq c_3\int_{0}^t\left\|P_{t-s}(|u(s)-v(s)|)\right\|_{L^2}ds \leq c_3\int_{0}^t\|z(s)\|_{L^2}ds.
\end{align}
Next, we consider $I_2(t)$. By H\"older's inequality, we have
\begin{align*}
    I_2(t) &\leq c_2\int_{0}^t\left\{\int_{0}^1\left[\int_{0}^1p_{t-s}(x,y)\log_+(|u(s,y)|\vee |v(s,y)|)|u(s,y) - v(s,y)|dy\right]^2dx\right\}^{\frac{1}{2}}ds \\&\leq
    c_2\int_{0}^t\left\{\int_{0}^1\left[\int_{0}^1p_{t-s}(x,y)|z(s,y)|^2dy\right]\left[\int_{0}^1p_{t-s}(x,y)\log_+^2(|u(s,y)|\vee |v(s,y)|)dy\right]dx\right\}^{\frac{1}{2}}ds \\ &\leq
    c_2\int_{0}^t\|z(s)\|_{L^2}\left\{\sup_{x\in[0,1]}\int_{0}^1p_{t-s}(x,y)\log_+^2(|u(s,y)|\vee |v(s,y)|)dy\right\}^{\frac{1}{2}}ds.
\end{align*}
According to \eqref{eq4}, we see that
\begin{align*}
&\sup_{x\in[0,1]}\int_{0}^1p_{t-s}(x,y)\log_+^2(|u(s,y)|\vee |v(s,y)|)dy  \\\leq&
\left\{\log_{+}C + \frac{1}{4}\log_{+}\left(\frac{1}{t-s}\right) + \log_{+}(\|u(s)\|_{L^2} + \|v(s)\|_{L^2})\right\}^2.
\end{align*}
Let $\tau_M := \inf\{t \geq0:\|u(t)\|_{L^2} \geq M\} \land \inf\{t \geq0:\|v(t)\|_{L^2} \geq M\}$. For $ t < \tau_M$, we have
\begin{align}\label{estimate2}
\begin{aligned}
    I_2(t) &\leq C\int_{0}^t\|z(s)\|_{L^2}ds + C\log_{+}(M)\int_{0}^t\|z(s)\|_{L^2} \\&+ \frac{c_2}{4}\int_{0}^t\|z(s)\|_{L^2}\log_{+}\left(\frac{1}{t-s}\right)ds.
\end{aligned}
\end{align}
Next, for any $\delta \in (0,\frac{\sqrt{2}}{2}e^{-1})$, we let $\tau^{\delta} = \inf \{t > 0: \|z(t)\|_{L^2} \geq \delta\}$. We now estimate $I_1(t)$ for $ t \leq \tau^{\delta}$,
\begin{align*}
    I_1(t) &= c_1\left\|\int_{0}^t\int_{0}^1p_{t-s}(x,y)|u(s,y) - v(s,y)|\log_+\left(\frac{1}{|u(s,y) - v(s,y)| }\right)dyds\right\|_{L^2_x} \\ &\leq
    c_1\int_{0}^t\left\|\int_{0}^1p_{t-s}(x,y)|u(s,y) - v(s,y)|\log_+\left(\frac{1}{|u(s,y) - v(s,y)|}\right)dy\right\|_{L^2_x}ds \\ &=
    c_1\int_{0}^t\left\{\int_{0}^1\left[\int_{0}^1p_{t-s}(x,y)|u(s,y)-v(s,y)|\log_+\left(\frac{1}{|u(s,y) - v(s,y)|}\right)dy\right]^2dx\right\}^{\frac{1}{2}}ds \\ &\leq
    c_1\int_{0}^t\left\{\int_{0}^1\int_{0}^1p_{t-s}(x,y)|u(s,y) - v(s,y)|^2\log_+^2\left(\frac{1}{|u(s,y) - v(s,y)|}\right)dydx\right\}^{\frac{1}{2}}ds \\ & \leq
    c_1\int_{0}^t\left\{\int_{0}^1|u(s,y) - v(s,y)|^2\log_+^2\left(\frac{1}{|u(s,y) - v(s,y)|}\right)dy\right\}^{\frac{1}{2}}ds.
\end{align*}
Let $f(z) = z(\log \frac{1}{z})^2,$ for $z\in (0,1]$.
Then $f'(z) = (2+\log z)\log z > 0$ when $z \in (0,e^{-2})$, so $f(z)$ increases for $z \in (0,e^{-2})$.
We know $f''(z) = \frac{2}{z}(1 + \log z )  < 0$ when $z < e^{-1}$. So $f$ is concave in $(0,e^{-1})$. Moreover,
\begin{align}\label{wang1}
    |f(z)| \leq 4|z|, \quad \text{for} \ z\in [e^{-2},1].
\end{align}
Next, we fix $s > 0$ and set $\Gamma := \{y \in[0,1]:|z(s,y)|^2 < e^{-2}\}$. In view of $\log_{+}\frac{1}{z} = 0$ for $z > 1$, we have
\begin{align}\label{wesnesday1}
\begin{aligned}[b]
    I_1(t) &\leq  c_1\int_{0}^t\left\{\int_{0}^1|u(s,y) - v(s,y)|^2\log_+^2\left(\frac{1}{|u(s,y) - v(s,y)|}\right)dy\right\}^{\frac{1}{2}}ds \\&\leq
    \frac{c_1}{2}\int_{0}^t\left\{\int_{\{y\in[0,1]:e^{-2} \leq |z(s,y)|^2 < 1\}}f(z^2(s,y))dy + \int_{\Gamma}f(z^2(s,y))dy\right\}^{\frac{1}{2}}ds.
\end{aligned}
\end{align}
Then by the Chebyshev inequality, we know that the Lebesgue measure $|\Gamma|$ satisfies
\begin{align*}
    |\Gamma| \geq 1 - e^2\|z(s)\|_{L^2}^2.
\end{align*}
This implies that $\forall\ t \leq \tau^{\delta}$,
\begin{align*}
    \frac{1}{|\Gamma|}\|z(t)\|_{L^2}^2 \leq \frac{\|z(t)\|_{L^2}^2}{1 - e^2\|z(t)\|_{L^2}^2} \leq \frac{1}{\delta^{-2} - e^{2}} \leq e^{-2}.
\end{align*}
By the Jensen inequality, \eqref{wang1} and the fact that $f$ is concave and increasing in $(0,e^{-2})$, \eqref{wesnesday1} becomes
\begin{align}\label{estimate3}
    \notag I_1(t) &\leq \frac{c_2}{2}\int_{0}^t\left\{\int_{\{y\in[0,1]:\frac{2}{e}\leq|z(s,y)|^2<1\}}4|z(s,y)|^2dy + |\Gamma|\int_{\Gamma}f(z^2(s,y))\frac{dy}{|\Gamma|}\right\}^{\frac{1}{2}}ds \\ \notag&\leq
    \frac{c_1}{2}\int_{0}^t\left\{4\|z(s)\|_{L^2}^2 + |\Gamma| f(\int_{\Gamma}z^2(s,y)\frac{dy}{|\Gamma|})\right\}^{\frac{1}{2}}ds \\ \notag&\leq
    \frac{c_1}{2}\int_{0}^t\left\{4\|z(s)\|_{L^2}^2 + |\Gamma|f(\frac{1}{|\Gamma|}\|z(s)\|_{L^2}^2)\right\}^{\frac{1}{2}}ds
    \\ \notag&\leq c_1\int_{0}^t\|z(s)\|_{L^2}ds + \frac{c_1}{2}\int_{0}^t\sqrt{|\Gamma|}\frac{\|z(s)\|_{L^2}}{\sqrt{|\Gamma|}}\log_+\left(\frac{|\Gamma|}{\|z(s)\|_{L^2}^2}\right)ds \\ & \leq
    c_1\int_{0}^t\|z(s)\|_{L^2}ds + c_1\int_{0}^t\|z(s)\|_{L^2}\log_+{\frac{1}{\|z(s)\|_{L^2}}}ds.
\end{align}
Combining \eqref{estimate1}, \eqref{estimate2} and \eqref{estimate3} together yields that for any $t \leq \tau^{\delta}_M := \tau_M \wedge \tau^{\delta}$,
\begin{align*}
    I(t) &\leq C\int_{0}^t\|z(s)\|_{L^2}ds + c_2\int_{0}^t\left(C_M + \frac{1}{4}\log_+\left(\frac{1}{t-s}\right)\right)\|z(s)\|_{L^2}ds \\&+ c_1\int_{0}^t\|z(s)\|_{L^2}ds + c_1\int_{0}^t\|z(s)\|_{L^2}\log_+\frac{1}{\|z(s)\|_{L^2}}ds.
\end{align*}
Hence there exists a constant $c > 0$ such that for $t \leq \tau^{\delta}_M$,
\begin{align*}
    \|z(t)\|_{L^2} \leq& \sup_{s \leq t}J(s) +  C_M\int_{0}^t\|z(s)\|_{L^2}ds + C\int_{0}^t(t-s)^{-\frac{1}{4}}\|z(s)\|_{L^2}ds  \\ +& C\int_{0}^t\|z(s)\|_{L^2}\log_+\frac{1}{\|z(s)\|_{L^2}}ds.
\end{align*}
Then, by Lemma 3.2, we have
\begin{align*}
    \|z(t)\|_{L^2} \leq C_T\sup_{s\leq t}J(s) + C_T \int_{0}^t\|z(s)\|_{L^2}\log_+\frac{1}{\|z(s)\|_{L^2}}ds.
\end{align*}
This implies
\begin{align*}
    \sup_{t\leq r\land \tau_M^\delta}\|z(t)\|_{L^2}  &\leq
    C_{M,T}\sup_{t \leq r\land \tau_M^\delta}J(t) + C_{M,T}\int_{0}^r\sup_{\rho \leq s\land \tau_M^\delta}\left(\|z(\rho)\|_{L^2}\log_+\frac{1}{\|z(\rho)\|_{L^2}}\right)ds .
\end{align*}
Denote
\begin{align*}
    F(r) = \mathbb{E}\left[\sup_{t\leq r \land \tau_M^\delta}\|z(t)\|_{L^2}\right].
\end{align*}
Let $h(x) = x\log_{+}\frac{1}{x}$ for $x \in (0,e^{-1})$. Then $h'(x) = \log\frac{1}{x} - 1 \geq 0$, this implies that $h(x)$ increases in $(0,e^{-1})$. And $h''(x) = -\frac{1}{x} \leq 0$, so $h(x)$ is concave in $(0,e^{-1})$.
Then we have
\begin{align*}
F(r) \leq  C_{M,T}\mathbb{E}\left[\sup_{t \leq r\land \tau_M^\delta}J(t)\right] + C_{M,T}\int_{0}^rF(s)\log_+\frac{1}{F(s)}ds.
\end{align*}
According to Lemma 4.2, for any $\ \epsilon > 0$, there exists a constant $C_\epsilon$ such that
\begin{align*}
    \mathbb{E}\left[\sup_{t \leq r\land \tau_M^\delta}J(t)\right] &= \mathbb{E}\left[\sup_{t \leq r\land \tau_M^\delta}\left\|\int_{0}^t\int_{0}^1p_{t-s}(x,y)[\sigma(u(s,y)) - \sigma(v(s,y))]W(ds,dy)\right\|_{L^2_x}\right]
    \\ &\leq
    \epsilon F(r) + C_{\epsilon,T}\int_{0}^rF(s)ds.
\end{align*}
So we get that
\begin{align*}
F(r) \leq \epsilon F(r) + C_{\epsilon,M,T}\int_{0}^rF(s)ds + C_{M,T}\int_{0}^rF(s)\log_+\frac{1}{F(s)}ds.
\end{align*}
According to Theorem 3.2, we conclude that $F \equiv 0$ on $[0,T]$. This implies
\begin{align*}
\mathbb{E}\left[\sup_{t \leq T\land \tau_M^\delta}\|u(t) - v(t)\|_{L^2}\right] = 0.
\end{align*}
Let $M \to \infty$, then $\tau_M \to \infty$. By Fatou's lemma, we have
\begin{align*}
\mathbb{E}\left[\sup_{t \leq T\land \tau^\delta}\|u(t) - v(t)\|_{L^2}\right] = 0.
\end{align*}
So $\sup_{t \leq T\land \tau^\delta}\|u(t) - v(t)\|_{L^2} = 0$, $\mathbb{P}-a.s$. This implies $\tau^\delta \geq T  $ $a.s.$ and
\begin{align*}
\sup_{t \leq T}\|u(t) - v(t)\|_{L^2} = 0.
\end{align*}
The proof of pathwise uniqueness is completed.
\end{proof}
\section{Appendix}
In this section, we will prove a result on the convergence of stochastic integral with respect to space-time white noise, which will be used in the proof of existence of probabilistically weak solutions. The result is similar to Lemma 4.3 in \cite{MR4904095} and Lemma 2.1 in \cite{MR2812364}.
\begin{Theorem}
    Assume that $H$ is a separable Hilbert space. Let $(\Omega,\mathcal{F,\mathbb{P}})$ be a probability space with a sequence of filtrations $\{\mathcal{F}_t^n\}_{t \geq 0}$ and $\{\mathcal{F}_t\}_{t \geq 0}$ satisfying the usual conditions. For each $n \in \mathbb{N}$, let $W_n$ be a Brownian sheet with respect to $\mathcal{F}_t^n$ and let $W$ be a Brownian sheet with respect to $\mathcal{F}_t$. Assume that for each $n \in \mathbb{N}$, $f_n$ is $H$-valued $\mathcal{F}_t^n$-adapted measurable processes such that
    \begin{align}
    \begin{aligned}\label{kk1}
    \sup_{n}\int_{0}^T\int_{0}^1\|f_n(s,y)\|_{H}^2dyds < \infty,\quad a.s.
    \end{aligned}
    \end{align}
    And $f$ is a $H$-valued $\mathcal{F}_t$-adapted measurable process such that
    \begin{align}\label{kk2}
    \begin{aligned}
    \int_{0}^T\int_{0}^1\|f(s,y)\|_{H}^2dyds < \infty, \quad a.s.
    \end{aligned}
    \end{align}
    If
    \begin{align}\label{eq10}
        \sup_{t\leq T,x\in[0,1]}|W_n(t,x) - W(t,x)| \xrightarrow[n \to \infty]{\mathbb{P}} 0 ,\quad
    \end{align}
    and
    \begin{align}\label{xxx1}
        \int_{0}^T\int_{0}^1\|f_{n}(s,y) - f(s,y)\|_{H}^2dyds \xrightarrow[n \to \infty]{\mathbb{P}}0.
    \end{align}
    Then we have
    \begin{align}\label{wang5}
        \sup_{t \leq T}\left\|\int_{0}^t\int_{0}^1f_n(s,y)W_n(ds,dy) - \int_{0}^t\int_{0}^1f(s,y)W(ds,dy)\right\|_{H} \xrightarrow[n \to \infty]{\mathbb{P}}0,
    \end{align}
    where $W_n(ds,dy)$ and $W(ds,dy)$ are the space-time white noise corresponding to the Brownian sheet $W_n$ and $W$ respectively.
\end{Theorem}
\begin{proof}
    Let $e_k(x) = \sqrt{2}\sin(k\pi x)$, which constitutes a set of orthonormal basis of $L^2([0,1])$. Then for each $k \geq 1$,
    \begin{align*}
        W_n^k(t)  = \int_{0}^t\int_{0}^1e_k(y)W_n(ds,dy) = -\int_{0}^1W_n(t,x)e'_k(x)dx
    \end{align*}
    is a one-dimensional standard Brownian motion. And similarly,
    \begin{align*}
        W^k(t) = \int_{0}^t\int_{0}^1e_k(y)W(ds,dy)= -\int_{0}^1W(t,x)e'_k(x)dx
    \end{align*}
    is also a one-dimensional standard Brownian motion for each $k \geq 1$. From \eqref{eq10}, we have
    \begin{align}\label{+2}
        \sup_{t \leq T}|W^k_n(t) - W^k(t)| \leq C\sup_{t \leq T}\int_{0}^1|W_n(t,x) - W(t,x)|dx \xrightarrow[n \to \infty]{\mathbb{P}}0.
    \end{align}
    Note that
    \begin{align*}
        \int_{0}^t\int_{0}^1f_n(s,y)W_n(ds,dy) &= \sum_{k=1}^\infty\int_{0}^t\left(\int_{0}^1f_n(s,y)e_k(y)dy\right)dW_n^k(s) \\&=\int_{0}^tF_n(s)dW_n(s),
    \end{align*}
    where $W_n(s) = \sum_{k=1}^{\infty}W_n^k(s)e_k$ is the corresponding cylindrical Wiener process on $L^2([0,1])$, and $F_n(s)$ is an Hilbert-Schmidt operator from $L^2([0,1])$ to $H$ and defined by
    \begin{align*}
        F_n(s)(g) := \int_{0}^1f_n(s,y)g(y)dy, \quad \forall\ g \in L^2([0,1]).
    \end{align*}
    Let $\{h_j\}_{j=1}^{\infty}$ be an orthonormal basis of $H$ and define $h_j\otimes e_k := h_j\langle e_k,\cdot\rangle_{L^2}$  for any $k,j \geq1$, then $h_j \otimes e_k$ is an orthonormal basis of $HS(L^2,H)$ (see Proposition B.0.7 in \cite{MR3410409}). Then the Hilbert-Schmidt norm of $F_n(s)$ can be computed as follows:
    \begin{align}\label{wang2}
    \begin{aligned}[b]
        \|F_n(s)\|_{HS(L^2,H)}^2 &= \sum_{k,j=1}^{\infty}\langle F_n(s),h_j \otimes e_k\rangle^2_{HS(L^2,H)} \\&= \sum_{k,j=1}^{\infty}\langle h_j,F_n(s)(e_k)\rangle_H^2 \\&= \sum_{j=1}^{\infty}\sum_{k=1}^{\infty}\left(\int_{0}^1\langle f_n(s,y),h_j\rangle_He_k(y)dy\right)^2 \\&=
        \sum_{j=1}^{\infty}\int_{0}^1\langle f_n(s,y),h_j\rangle_H^2dy \\&=
        \int_{0}^1\|f_n(s,y)\|_H^2dy,
    \end{aligned}
    \end{align}
    Similarly, we can define $F(s)$ and $W(s)$ and have
    \begin{align*}
        \int_{0}^t\int_{0}^1f(s,y)W(ds,dy) = &\sum_{k=1}^{\infty}\int_{0}^t\left(\int_{0}^1f(s,y)e_{k}(y)dy\right)dW^{k}(s) \\=& \int_{0}^tF(s)dW(s).
    \end{align*}
    By \eqref{xxx1} and \eqref{wang2}, we have
    \begin{align}\label{+3}
        \int_{0}^T\|F_n(s)-F(s)\|_{HS(L^2,H)}^2ds = \int_{0}^T\int_{0}^1\|f_n(s,y)-f(s,y)\|_{H}^2dyds \xrightarrow[n \to \infty]{\mathbb{P}}0.
    \end{align}
    As the conditions of Lemma 4.3 in \cite{MR4904095} are fulfilled by \eqref{+2} and \eqref{+3}, it follows from Lemma 4.3 in \cite{MR4904095} that \eqref{wang5} holds.
\end{proof}
\vskip 0.3cm
	\noindent {\large{\bf{{Acknowledgement}}}}

This work is partially supported by the National Key R\&D Program of China (No. 2022
YFA1006001), the National Natural Science Foundation of China (No. 12131019, 12571158, 12371151, 11721101), and the Fundamental Research Funds for the Central Universities(No. WK3470000031, WK0010000081).
\bibliographystyle{plain}
\bibliography{ref.bib}
\end{document}